\documentclass[12pt,reqno]{amsart}
\usepackage[colorlinks=true, pdfstartview=FitV, linkcolor=blue, citecolor=blue, urlcolor=blue]{hyperref}
\usepackage{amssymb,amsmath, amscd}
\usepackage{ragged2e}
\usepackage{times, verbatim}
\usepackage{mathdots}
\usepackage{array}
\usepackage{graphicx}
\usepackage[english]{babel}
 \usepackage[usenames, dvipsnames]{color}
\usepackage{amsmath,amssymb,amsfonts}
\usepackage{enumerate, enumitem}
\usepackage{MnSymbol}
\usepackage{anysize}
\usepackage{enumitem}
\usepackage{bigints}
\usepackage{braket}
\usepackage[english]{babel}
\usepackage{blindtext}
\usepackage{mathtools}
\usepackage{graphicx}
\usepackage{tikz-cd}
\usepackage[a4paper,verbose]{geometry}
\marginsize{2.5cm}{2.5cm}{2.5cm}{2.5cm}
\input xy
\xyoption{all}
\usepackage{pb-diagram}
\usepackage[all]{xy}
\input xy
\xyoption{all}

\DeclareFontFamily{OT1}{rsfs}{}
\DeclareFontShape{OT1}{rsfs}{n}{it}{<-> rsfs10}{}
\DeclareMathAlphabet{\mathscr}{OT1}{rsfs}{n}{it}

\numberwithin{equation}{section}
\numberwithin{equation}{subsection}
\renewcommand*{\theequation}{%
  \ifnum\value{subsection}=0 %
    \thesection
  \else
    \thesubsection
  \fi
  .\arabic{equation}%
}

\newtheorem{thm}[equation]{Theorem}
\newtheorem{prop}[equation]{Proposition}

\newtheorem{lemma}[equation]{Lemma}
\newtheorem{defn}[equation]{Definition}

\newtheorem{remark}[equation]{Remark}

\newtheorem{corollarySubsec}[equation]{Corollary}

\theoremstyle{definition}
\newtheorem*{funding}{Funding}

\theoremstyle{remark}
\newtheorem*{acknowledgements}{Acknowledgments}

\begin{document}
\theoremstyle{plain}

\newcommand{\bigboxplus}{
	\mathop{
		\vphantom{\bigoplus} 
		\mathchoice
		{\vcenter{\hbox{\Resizebox{\widthof{$\displaystyle\bigoplus$}}{!}{$\boxplus$}}}}
		{\vcenter{\hbox{\Resizebox{\widthof{$\bigoplus$}}{!}{$\boxplus$}}}}
		{\vcenter{\hbox{\Resizebox{\widthof{$\scriptstyle\oplus$}}{!}{$\boxplus$}}}}
		{\vcenter{\hbox{\Resizebox{\widthof{$\scriptscriptstyle\oplus$}}{!}{$\boxplus$}}}}
	}\displaylimits 
}

\newcommand{\Hecke}{\mathcal{H}}
\newcommand{\GL}{\operatorname{GL}}
\newcommand{\Liea}{\mathfrak{a}}
\newcommand{\Cmg}{C_{\mathrm{mg}}}
\newcommand{\Cinftyumg}{C^{\infty}_{\mathrm{umg}}}
\newcommand{\Cfd}{C_{\mathrm{fd}}}
\newcommand{\Cinftyfd}{C^{\infty}_{\mathrm{ufd}}}
\newcommand{\sspace}{\Gamma \backslash G}
\newcommand{\PP}{\mathcal{P}}
\newcommand{\bfP}{\mathbf{P}}
\newcommand{\bfQ}{\mathbf{Q}}
\newcommand{\Siegel}{\mathfrak{S}}
\newcommand{\g}{\mathfrak{g}}
\newcommand{\A}{\mathbb{A}}
\newcommand{\Q}{\mathbb{Q}}
\newcommand{\Gm}{\mathbb{G}_m}
\newcommand{\Nm}{\mathbb{N}m}
\newcommand{\ii}{\mathfrak{i}}
\newcommand{\II}{\mathfrak{I}}

\newcommand{\kk}{\mathfrak{k}}
\newcommand{\nn}{\mathfrak{n}}
\newcommand{\tF}{\widetilde{F}}
\newcommand{\p}{\mathfrak{p}}
\newcommand{\m}{\mathfrak{m}}
\newcommand{\bb}{\mathfrak{b}}
\newcommand{\Ad}{{\rm Ad}\,}
\newcommand{\ttt}{\mathfrak{t}}
\newcommand{\frakt}{\mathfrak{t}}
\newcommand{\U}{\mathcal{U}}
\newcommand{\Z}{\mathbb{Z}}
\newcommand{\bfG}{\mathbf{G}}
\newcommand{\bfT}{\mathbf{T}}
\newcommand{\R}{\mathbb{R}}
\newcommand{\ST}{\mathbb{S}}
\newcommand{\h}{\mathfrak{h}}
\newcommand{\bC}{\mathbb{C}}
\newcommand{\C}{\mathbb{C}}
\newcommand{\N}{\mathbb{N}}
\newcommand{\qH}{\mathbb {H}}
\newcommand{\temp}{{\rm temp}}
\newcommand{\Hom}{{\rm Hom}}
\newcommand{\Aut}{{\rm Aut}}
\newcommand{\rk}{{\rm rk}}
\newcommand{\Ext}{{\rm Ext}}
\newcommand{\End}{{\rm End}\,}
\newcommand{\Ind}{{\rm Ind}}
\newcommand{\ind}{{\rm ind}}
\newcommand{\Irr}{{\rm Irr}}
\def\circG{{\,^\circ G}}
\def\M{{\rm M}}
\def\diag{{\rm diag}}
\def\Ad{{\rm Ad}}
\def\As{{\rm As}}
\def\wG{{\widehat G}}
\def\G{{\rm G}}
\def\SL{{\rm SL}}
\def\PSL{{\rm PSL}}
\def\GSp{{\rm GSp}}
\def\PGSp{{\rm PGSp}}
\def\Sp{{\rm Sp}}
\def\St{{\rm St}}
\def\GU{{\rm GU}}
\def\SU{{\rm SU}}
\def\U{{\rm U}}
\def\GO{{\rm GO}}
\def\GL{{\rm GL}}
\def\PGL{{\rm PGL}}
\def\GSO{{\rm GSO}}
\def\GSpin{{\rm GSpin}}
\def\GSp{{\rm GSp}}

\def\Gal{{\rm Gal}}
\def\SO{{\rm SO}}
\def\O{{\rm  O}}
\def\Sym{{\rm Sym}}
\def\sym{{\rm sym}}
\def\St{{\rm St}}
\def\Sp{{\rm Sp}}
\def\tr{{\rm tr\,}}
\def\ad{{\rm ad\, }}
\def\Ad{{\rm Ad\, }}
\def\rank{{\rm rank\,}}

\def\Ext{{\rm Ext}}
\def\Hom{{\rm Hom}}
\def\Alg{{\rm Alg}}
\def\GL{{\rm GL}}
\def\SO{{\rm SO}}
\def\G{{\rm G}}
\def\U{{\rm U}}
\def\St{{\rm St}}
\def\Wh{{\rm Wh}}
\def\RS{{\rm RS}}
\def\ind{{\rm ind}}
\def\Ind{{\rm Ind}}
\def\csupp{{\rm csupp}}

\title{Finite-Sum Realization of Archimedean Asai and Exterior-Square $L$-Factors}
\author[Yeongseong Jo and Akash Yadav]{Yeongseong Jo$^{\dagger}$ and Akash Yadav$^{\dagger}$}
\thanks{$^{\dagger}$All authors made equal contributions and were listed alphabetically.}
\address{(Y. Jo) Department of Mathematics Education / Human-Centered Artificial Intelligence Research Institute, Ewha Womans University, Seoul 03760, Republic of Korea}
\email{yeongseong.jo@ewha.ac.kr, yeongseongjo@outlook.com}
\address{(A. Yadav) Human-Centered Artificial Intelligence Research Institute, Ewha Womans University, Seoul 03760, Republic of Korea}
\email{akaschampion@ewha.ac.kr, akaschampion@gmail.com}


\begin{abstract}
We prove that the archimedean Asai \(L\)-factor attached to an
irreducible generic representation of \(\GL_n(\mathbb{C})\) can be
expressed as a finite sum of Flicker local zeta integrals. For an
archimedean local field \(F\), we also prove that the exterior-square
\(L\)-factor attached to an irreducible generic representation of
\(\GL_m(F)\) can be expressed as a finite sum of Jacquet--Shalika local
zeta integrals.
\end{abstract}
\maketitle

\begin{section}{Introduction}\label{s1}

The Rankin--Selberg method realizes several automorphic \(L\)-functions
as global zeta integrals involving automorphic forms, often together with
Eisenstein series. To study their global analytic properties, one must understand the associated local zeta integrals and their relation to the local $L$-factors at every place.

In the archimedean setting, Beuzart--Plessis
\cite[Theorem~1.1]{BP2021} established the basic analytic properties of
Flicker's integral representation \cite{Fli1993} of the Asai
\(L\)-function. Jiang--Liu--Sun--Tian
\cite[Theorem~2.2]{JLST2025} obtained analogous results for the
Jacquet--Shalika integral representation \cite{JS1990} of the
exterior-square \(L\)-function. More precisely, in both settings, the authors proved convergence in a right half-plane, meromorphic continuation, and a functional equation for the local zeta integrals. They also showed that the quotient of each integral by the corresponding local $L$-factor is holomorphic. Consequently, the order of any pole of a local zeta
integral is at most that of the corresponding \(L\)-factor.

These results do not show, however, that the archimedean \(L\)-factor itself
can be recovered from the local zeta integrals. In the Asai case,
Beuzart--Plessis wrote
\cite[p.~5]{BP2021}:
\begin{quote}
``Actually, we even expect something stronger: there should exist finite
families \(W_i\in\mathcal{W}(\pi,\psi)\) and
\(\phi_i\in\mathcal{S}(F^n)\) such that
\[
L(s,\pi,\As)=\sum_i Z(s,W_i,\phi_i).
\]
However, in the archimedean case this result seems harder to establish
and in any case unreachable by the method developed in this paper.''
\end{quote}

In view of the analogous analytic theory of Jacquet--Shalika integrals, it is
natural to expect a corresponding finite-sum identity for the exterior-square
\(L\)-factor.

The goal of this paper is to establish these finite-sum identities in
both settings. In
Theorems~\ref{main1} and~\ref{main2}, we prove that the archimedean Asai and
exterior-square \(L\)-factors attached to irreducible generic
representations can be expressed as finite sums of Flicker and
Jacquet--Shalika local zeta integrals, respectively. 
The corresponding finite-sum identity for standard Rankin--Selberg integrals was established by Jacquet \cite[Theorem~2.6]{Jac2009}, building on the foundational work of Jacquet and Shalika \cite{JS1989}.

Before stating our
main results, we introduce the necessary notation. Let \(\pi\) be an irreducible generic representation of
\(\mathrm{GL}_n(\mathbb{C})\). Fix a nontrivial additive character
\(\psi_\mathbb{C} \colon \mathbb{C} \to \mathbb{C}^{\times}\) that is trivial on \(\mathbb{R}\). Denote by \(\mathcal{W}(\pi,\psi_\mathbb{C})\) the Whittaker model of \(\pi\), and by
\(\mathcal{S}(\mathbb{R}^n)\) the space of Schwartz functions on \(\mathbb{R}^n\).
For \(W \in \mathcal{W}(\pi,\psi_\mathbb{C})\), \(\phi \in \mathcal{S}(\mathbb{R}^n)\), and \(s \in \mathbb{C}\), we define,
whenever convergent, the associated Flicker integral \cite{Fli1993,FZ1995} 
\[
  I(s,W,\phi)
  :=
  \int_{N_n(\mathbb{R})\backslash \mathrm{GL}_n(\mathbb{R})}
  W(g)\phi(e_ng)|\det g|_{\mathbb{R}}^s\,dg,
\]
where \(N_n\) is the standard upper triangular unipotent subgroup and
\(e_n=(0,\ldots,0,1)\in \mathbb{R}^n\). It converges for \(\mathrm{Re}(s) \gg 0\) and extends to a meromorphic function on $\mathbb{C}$ \cite{BP2021}. Our first main theorem asserts that the Asai \(L\)-factor
\(L(s,\pi,\As)\), defined via the local Langlands correspondence, can be
expressed as a finite sum of Flicker integrals.

\begin{thm}\label{main1}
Let \(\pi\) be an irreducible generic representation of
\(\mathrm{GL}_n(\mathbb{C})\). Then there exist an integer \(r\geq 1\),
\(W_i\in\mathcal{W}(\pi,\psi_{\mathbb{C}})\), and
\(\phi_i\in\mathcal{S}(\mathbb{R}^n)\), \(1\leq i\leq r\), such that
\[
L(s,\pi,\As)=\sum_{i=1}^r I(s,W_i,\phi_i).
\]
\end{thm}

Theorem~\ref{main1} is the archimedean analogue of Matringe's
non-archimedean result \cite[Theorem~5.3]{Mat2011}. For \(n=2\), Chen, Cheng, and Ishikawa
\cite[Theorem~3.1(3)]{CCI2020} established such a finite-sum identity
using families of \(\mathrm{U}(2)\)-finite vectors and standard Schwartz
functions. Theorem~\ref{main1} extends the existence of such a finite-sum
expression from \(n=2\) to arbitrary \(n\).

\par
We next consider the exterior-square \(L\)-factor. Let \(F\) be an
archimedean local field and let \(m\geq2\), written as \(m=2n\) or \(m=2n+1\) according to its parity. Let \(\pi\) be an irreducible generic representation of
\(\mathrm{GL}_m(F)\). Fix a nontrivial additive character
\(\psi_F\colon F\to\mathbb{C}^{\times}\). 
We write \(M_n(F)\) for the space of \(n\times n\) matrices and \(V_n(F)\) for its subspace of upper triangular matrices. Let \(\sigma_{2n}\) and \(\sigma_{2n+1}\) be the permutation matrices defined in Section~\ref{s2}.

For \(m=2n\), \(W\in\mathcal{W}(\pi,\psi_F)\),
\(\phi\in\mathcal{S}(F^n)\), and $s\in\mathbb{C}$, we define,
whenever convergent, the Jacquet--Shalika integral
\cite{JS1990} 
\[
\begin{aligned}
  J(s,W,\phi)
  &:=
  \int_{N_n(F)\backslash\mathrm{GL}_n(F)}
  \int_{V_n(F)\backslash M_n(F)}
  W\left(
    \sigma_{2n}
    \begin{pmatrix}
      I_n & X\\
      0   & I_n
    \end{pmatrix}
    \begin{pmatrix}
      g & 0\\
      0 & g
    \end{pmatrix}
  \right) \\
  &\qquad\qquad\times
  \psi_F(-\operatorname{tr}X)\phi(e_ng)
  |\det g|_F^s\,dX\,dg.
\end{aligned}
\]
If \(m=2n+1\), then, for \(W\in\mathcal{W}(\pi,\psi_F)\), the
corresponding integral is
\[
\begin{aligned}
  J(s,W)
  &:=
  \int_{N_n(F)\backslash\mathrm{GL}_n(F)}
  \int_{V_n(F)\backslash M_n(F)}
  W\left(
    \sigma_{2n+1}
    \begin{pmatrix}
      I_n & X   & 0\\
      0   & I_n & 0\\
      0   & 0   & 1
    \end{pmatrix}
    \begin{pmatrix}
      g & 0 & 0\\
      0 & g & 0\\
      0 & 0 & 1
    \end{pmatrix}
  \right) \\
  &\qquad\qquad\times
  \psi_F(-\operatorname{tr}X)|\det g|_F^{s-1}\,dX\,dg.
\end{aligned}
\]
These integrals converge for \(\mathrm{Re}(s) \gg 0\) and extend to meromorphic functions on $\mathbb{C}$ \cite{JS1990, Bel2011,JLST2025}.
Let $L(s,\pi,\wedge^2)$ denote the exterior-square $L$-factor of $\pi$ defined via the local Langlands correspondence. Our second main theorem is the following.

\begin{thm}\label{main2}
Let \(\pi\) be an irreducible generic representation of
\(\mathrm{GL}_m(F)\). There exist an integer \(r\geq 1\), a finite
family of Whittaker functions
\(\{W_i\}_{i=1}^r\subset\mathcal{W}(\pi,\psi_F)\), and, when
\(m=2n\), a finite family of Schwartz functions
\(\{\phi_i\}_{i=1}^r\subset\mathcal{S}(F^n)\), such that
\[
L(s,\pi,\wedge^2)=
\begin{cases}
\displaystyle\sum_{i=1}^r J(s,W_i,\phi_i)
  & \text{if } m=2n,\\[6pt]
\displaystyle\sum_{i=1}^r J(s,W_i)
  & \text{if } m=2n+1.
\end{cases}
\]
\end{thm}

For the corresponding non-archimedean results, see \cite{KR2012,Jo2020}.

Theorems~\ref{main1} and~\ref{main2} have applications in both
local and global settings. Locally, Theorem~\ref{main1} yields a
characterization of distinction: an irreducible generic representation
$\pi$ of $\mathrm{GL}_n(\mathbb{C})$ is
$\mathrm{GL}_n(\mathbb{R})$-distinguished if and only if $s=0$ is an
exceptional pole of $L(s,\pi,\As)$ with level $0$; see
Theorem~\ref{thm:distinguished-exceptional}. Globally, combining Theorems~\ref{main1} and~\ref{main2} with their non-archimedean counterparts, we express the global Asai and exterior-square $L$-functions as finite sums of global Flicker and Jacquet--Shalika integrals, respectively; see Theorems~\ref{sum-global} and~\ref{sum-global-exterior-square}.

We sketch the proof of Theorem~\ref{main1}; the proof of
Theorem~\ref{main2} is similar. The main difficulty is that Jacquet's
inductive argument \cite{Jac2009} does not apply directly. Indeed, for
semisimple representations \(\tau_1\) and \(\tau_2\) of Weil group
\(W_{\mathbb C}\), we have
\[
\As(\tau_1\oplus\tau_2)
=
\As(\tau_1)\oplus\As(\tau_2)\oplus
\operatorname{Ind}_{W_{\mathbb C}}^{W_{\mathbb R}}
\bigl(\tau_1\otimes\tau_2^c\bigr),
\]
where \(\tau_2^c(z)=\tau_2(\overline z)\). The last summand is of
Rankin--Selberg type. Therefore, a direct induction would require a
comparison between local Rankin--Selberg and Flicker integrals. We avoid such a comparison by applying induction to two
overlapping subparameters.

Write the Langlands parameter of \(\pi\) as
\[
\sigma=\chi_1\oplus\cdots\oplus\chi_n,
\qquad
\sigma^+=\chi_2\oplus\cdots\oplus\chi_n,
\qquad
\sigma^-=\chi_1\oplus\cdots\oplus\chi_{n-1},
\]
where the \(\chi_i\) are characters of 
\(W_{\mathbb{C}}\), ordered so that their exponents \(e_i\) satisfy
$e_1\leq\cdots\leq e_n.$ 
Then
\[
\As(\sigma)
=
\rho_1\oplus\As(\sigma^+)
=
\As(\sigma^-)\oplus\eta_2,
\]
where
\[
\begin{aligned}
\rho_1&=\As(\chi_1)\oplus
\operatorname{Ind}_{W_{\mathbb{C}}}^{W_{\mathbb{R}}}
\bigl(\chi_1\otimes(\sigma^+)^c\bigr),\\
\eta_2&=\As(\chi_n)\oplus
\operatorname{Ind}_{W_{\mathbb{C}}}^{W_{\mathbb{R}}}
\bigl(\sigma^-\otimes\chi_n^c\bigr).
\end{aligned}
\]

In Section~\ref{s3}, we associate to each semisimple Weil-group representation \(\rho\) an \(L\)-space \(\mathcal{L}(\rho)\). It consists of functions of the form \(L(s,\rho)h(s)\), where \(h\) is entire and \(L(s,\rho)h(s)\) is rapidly decreasing in vertical strips. We prove that
\[
\mathcal{L}\bigl(\As(\sigma)\bigr)
=
\mathcal{L}\bigl(\As(\sigma^-)\bigr)
+
\frac{L(s,\rho_1)}{L(1-s,\rho_1^\vee)}
\mathcal{L}\bigl(\As(\sigma^+)\bigr),
\]
which reduces the proof to the two \((n-1)\)-dimensional subparameters \(\sigma^-\) and \(\sigma^+\).

The paper is organized as follows. In Section~\ref{s2}, we fix the
notation and recall the necessary definitions. In Section~\ref{s3}, we
establish the key decomposition identity for $L$-spaces used in the
proofs of Theorems~\ref{main1} and~\ref{main2}. In
Sections~\ref{s4} and~\ref{s5}, we prove Theorems~\ref{main1}
and~\ref{main2}, respectively. We present the local applications in
Section~\ref{s6} and the global applications in Section~\ref{s7}.
\end{section} 

\section{Preliminaries}\label{s2}

\subsection{Basic notation}

Let \(F\) be an archimedean local field and let \(n\geq1\) be an
integer. We write \(|\cdot|_F\) for the normalized absolute value on
\(F\). For the Jacquet--Shalika integrals, we take \(m=2n\) or \(m=2n+1\), according to the parity under consideration. Set \(G_n=\GL_n\), and let \(B_n\) and \(N_n\) denote its upper
triangular and upper triangular unipotent subgroups, respectively. Let \( P_n \) be the mirabolic subgroup of \( G_n \), consisting of matrices that stabilize the row vector \( e_n = (0, \dots, 0, 1) \) under the right multiplication action. Let $K_n=\operatorname{O}(n)$ be the maximal compact subgroup of $G_n(\mathbb{R})$. For any commutative ring $R$, we write $M_n(R)$ for the algebra of $n \times n$ matrices over $R$, and $V_n(R)$ for its submodule of upper triangular matrices. Let $\mathbb{S}^1$ denote the multiplicative group of all complex numbers with absolute value $1$. Let $w_n\in G_n(F)$ denote the long Weyl element, that is,
\[
w_n=
\begin{pmatrix}
&&1\\
&\iddots&\\
1&&
\end{pmatrix}.
\]
Let \(\sigma_{2n}\) and \(\sigma_{2n+1}\) be the permutation matrices
given by
\[
\sigma_{2n}
=
\left(
\begin{array}{cccc|cccc}
1 & 2 & \cdots & n
  & n+1 & n+2 & \cdots & 2n \\
1 & 3 & \cdots & 2n-1
  & 2 & 4 & \cdots & 2n
\end{array}
\right)
\]
and
\[
\sigma_{2n+1}
=
\left(
\begin{array}{cccc|ccccc}
1 & 2 & \cdots & n
  & n+1 & n+2 & \cdots & 2n & 2n+1 \\
1 & 3 & \cdots & 2n-1
  & 2 & 4 & \cdots & 2n & 2n+1
\end{array}
\right).
\]
We also set
\[
\tau_{2n}
=
\begin{pmatrix}
0&I_n\\
I_n&0
\end{pmatrix}
\qquad\text{and}\qquad
\tau_{2n+1}
=
\begin{pmatrix}
0&I_n&0\\
I_n&0&0\\
0&0&1
\end{pmatrix}.
\]

Fix a nontrivial unitary additive character
\(\psi_F\colon F\to\mathbb S^1\). In the Asai setting, the character
\(\psi_{\mathbb C}\) is additionally assumed to be trivial on
\(\mathbb R\). We use the same notation for the
generic character of $N_n(F)$ given by
\[
\psi_F(u)=\psi_F\left(\sum_{i=1}^{n-1}u_{i,i+1}\right), \qquad u\in N_n(F).
\]

Let \(\tau\in\mathbb C^\times\) be the unique element satisfying
\[
\psi_{\mathbb C}(z)
=\psi_{\mathbb R}(\operatorname{Tr}_{\mathbb C/\mathbb R}(\tau z)),
\qquad z\in\mathbb C,
\]
where $\operatorname{Tr}_{\mathbb{C}/\mathbb{R}}$ denotes the trace of the extension $\mathbb{C}/\mathbb{R}$.

We let $\mathcal{S}(F^n)$ be the space of Schwartz functions on $F^n$. We denote by $\phi \mapsto \widehat{\phi}$ the Fourier transform on $F^n$ defined as follows: for every $\phi \in \mathcal{S}(F^n)$ we have $$ \widehat{\phi}\left(x_1, \ldots, x_n\right)=\int_{F^n} \phi\left(y_1, \ldots, y_n\right) \psi_{F}\left(x_1 y_1+\ldots+x_n y_n\right) d y_1 \ldots d y_n $$ for all $\left(x_1, \ldots, x_n\right) \in F^n$, where the measure of integration is chosen so that $\widehat{\widehat{\phi}}(v)=\phi(-v)$. When we wish to emphasize the dependence on the additive character, we
write \(\mathcal{F}_{\psi_F}(\phi)\) for \(\widehat{\phi}\).

By a representation of $G_n(F)$, we mean a smooth admissible
Fr\'echet representation of moderate growth in the sense of
Casselman--Wallach. We denote the set of isomorphism classes of irreducible representations of \( G_n(F) \) by \( \operatorname{Irr}(G_n(F)) \). For $\pi\in\operatorname{Irr}(G_n(F))$, we denote its contragredient by $\pi^\vee$ and its central character by $\omega_\pi$.  A representation $\pi\in\operatorname{Irr}(G_n(F))$ is called \emph{generic} if
$$\operatorname{Hom}_{N_n(F)}(\pi,\psi_F)\neq 0.$$ We denote the subset of generic representations in \(\operatorname{Irr}(G_n(F))\) by \(\operatorname{Irr}_{\mathrm{gen}}(G_n(F))\). For such a
representation, we denote its Whittaker model by
$\mathcal{W}(\pi,\psi_F)$. For a smooth manifold \(X\), we denote by
\(C_c^\infty(X)\) the space of smooth compactly supported
complex-valued functions on \(X\). Let \(R\) denote the right translation action of
\(G_n(F)\) on \(\mathcal{W}(\pi,\psi_{F})\), given by
\[
(R(g)W)(x)=W(xg)\>\>\>\forall\>x \in G_n(F).
\] For \(f\in C_c^\infty(G_n(F))\), we also write
\[
R(f)W
=
\int_{G_n(F)} f(g)R(g)W\,dg.
\]
For $W\in\mathcal{W}(\pi,\psi_F)$, define
\[
\widetilde{W}(g)=W\left(w_n\,{}^tg^{-1}\right),
\qquad g\in G_n(F).
\]
Then $\widetilde{W}\in
\mathcal{W}(\pi^{\vee},\psi_F^{-1})$, and the map
$W\mapsto\widetilde{W}$ induces an isomorphism of topological vector
spaces
\[
\mathcal{W}(\pi,\psi_F)
\cong
\mathcal{W}(\pi^{\vee},\psi_F^{-1}).
\]

Let $W_F$ denote the Weil group of $F$. If $F=\mathbb{C}$, then
$W_F=\mathbb{C}^{\times}$. If $F=\mathbb{R}$, then
$$
W_F=\mathbb{C}^{\times}\sqcup j\mathbb{C}^{\times},
$$
where $j^2=-1$ and $jwj^{-1}=\overline{w}
$ for $w\in\mathbb{C}^{\times}.
$
An \emph{admissible representation} of \(W_F\) is a continuous
semisimple representation
\[
\sigma \colon W_F\longrightarrow \operatorname{GL}(V_\sigma),
\]
where \(V_\sigma\) is a finite-dimensional complex vector space. We
denote by \(L(s,\sigma)\) and \(\varepsilon(s,\sigma,\psi_F)\) the local
\(L\)- and \(\varepsilon\)-factors attached to \(\sigma\), as in
\cite{Tate1979}, and define
\[
\gamma(s,\sigma,\psi_F)
=
\varepsilon(s,\sigma,\psi_F)
\frac{L(1-s,\sigma^{\vee})}{L(s,\sigma)},
\]
where \(\sigma^{\vee}\) denotes the contragredient of \(\sigma\).
Let $\eta$ be the quadratic character of $W_\mathbb{R}$ associated to the extension $\mathbb{C}/\mathbb{R}$ and set
\[
\lambda_{\mathbb{C}/\mathbb{R}}\left(\psi_\R\right)=\varepsilon\left(\frac{1}{2}, \eta, \psi_\R\right) .\]

An irreducible admissible representation of \(W_F\) is said to be
\emph{normalized} if its restriction to
\(\mathbb{R}_{>0}^{\times}\subset W_F\) is trivial. We may write
\[
\sigma
=
\bigoplus_{i=1}^r
\sigma_i\otimes |\cdot|_F^{u_i},
\]
where the \(\sigma_i\) are normalized irreducible representations of
\(W_F\), and \(u_1,\ldots,u_r\in\mathbb{C}\) are ordered so that
\[
\operatorname{Re}(u_1)
\leq
\operatorname{Re}(u_2)
\leq
\cdots
\leq
\operatorname{Re}(u_r).
\]
Set
\[
D
=
\bigl((\sigma_1,u_1),\ldots,(\sigma_r,u_r)\bigr),
\]
and denote the corresponding induced representation by
\((\pi_D,I_D)\). Although the above decomposition of \(\sigma\) is not
unique, the equivalence class of \((\pi_D,I_D)\) depends only on
\(\sigma\). We denote this equivalence class by
\((\pi_\sigma,I_\sigma)\) and call it the generic induced
representation attached to \(\sigma\).

Fix a nonzero \(\psi_F\)-Whittaker functional
\(\lambda_\sigma\) on \(I_\sigma\). For \(v\in I_\sigma\), set
\[
W_v(g)=\lambda_\sigma\bigl(\pi_\sigma(g)v\bigr),
\qquad g\in G_n(F),
\]
and write
\[
\mathcal{W}(\pi_\sigma,\psi_F)
=
\{W_v:v\in I_\sigma\}.
\]
We call this the Whittaker realization of \((\pi_\sigma,I_\sigma)\).
When \(\pi_\sigma\) is irreducible, it is the usual Whittaker model. For this realization, the formula
\(\widetilde W(g)=W\left(w_n\,{}^tg^{-1}\right)\) gives a bijection
\[
\mathcal{W}(\pi_\sigma,\psi_F)
\longrightarrow
\mathcal{W}(\pi_{\sigma^\vee},\psi_F^{-1}).
\]

The numbers $\mathrm{Re}(u_1),\ldots,\mathrm{Re}(u_r)$ are called the \emph{exponents} of $\sigma$. The exponents of
$\sigma^{\vee}$ are the negatives of the exponents of $\sigma$.
For two admissible representations $\sigma$ and $\sigma'$ of $W_F$,
we write $\sigma\preceq\sigma'$ if the largest exponent of $\sigma$
is less than or equal to the smallest exponent of $\sigma'$.

For a meromorphic function \(f\) on \(\mathbb{C}\), we denote its set
of poles by \(\mathcal{P}(f)\). If $s_0\in\mathcal{P}(L(s,\sigma))$, there is an exponent $v$ of
$\sigma$ such that
$\mathrm{Re}(s_0)+v\leq 0.$

\subsection{The Asai representation}

Let \(\sigma\colon W_{\mathbb{C}}\to\GL(V)\) be an admissible
representation. The Asai representation
\[
\As(\sigma)\colon W_{\mathbb{R}}\to\GL(V\otimes V)
\]
is defined on pure tensors by
\[
\begin{aligned}
\As(\sigma)(z)(u\otimes v)
&=\sigma(z)u\otimes\sigma(jzj^{-1})v,
&& z\in W_{\mathbb{C}},\\
\As(\sigma)(j)(u\otimes v)
&=v\otimes\sigma(j^2)u,
\end{aligned}
\qquad u,v\in V.
\] Then \(\As(\sigma)\) is an admissible
representation of \(W_{\mathbb{R}}\). We write
\[
\begin{aligned}
L(s,\sigma,\As)
&=L(s,\As(\sigma)),\\
\varepsilon(s,\sigma,\As,\psi_{\mathbb{R}})
&=\varepsilon(s,\As(\sigma),\psi_{\mathbb{R}}).
\end{aligned}
\]
Suppose that
\[
\sigma=\sigma_1\oplus\cdots\oplus\sigma_k
\]
is a direct sum of admissible representations of
$W_{\mathbb{C}}$. Then
\[
\As(\sigma)
\simeq
\bigoplus_{i=1}^k\As(\sigma_i)
\oplus
\bigoplus_{1\leq i<j\leq k}
\operatorname{Ind}_{W_{\mathbb{C}}}^{W_{\mathbb{R}}}
\bigl(\sigma_i\otimes\sigma_j^c\bigr),
\]
where $\sigma_j^c$ is the representation of $W_{\mathbb{C}}$ defined by
$
\sigma_j^c(z)
=
\sigma_j(jzj^{-1}).
$

Let $\pi\in\operatorname{Irr}(G_n(\mathbb C))$. The local Langlands
correspondence for $G_n(\mathbb C)$ \cite[Theorem~5]{Kna1994} associates with $\pi$ an
$n$-dimensional admissible representation
\[
\sigma_\pi\colon W_{\mathbb C}\to\GL(V).
\]
We set
\[
L(s,\pi,\As)
:=
L\bigl(s,\As(\sigma_\pi)\bigr),
\]
\[
\varepsilon(s,\pi,\As,\psi_{\mathbb R})
:=
\varepsilon\bigl(s,\As(\sigma_\pi),\psi_{\mathbb R}\bigr)
\]
and call them the Asai $L$-function and $\varepsilon$-factor
of $\pi$, respectively.

\subsection{The exterior-square representation}

Let $\sigma\colon W_F\to\GL(V)$ be an admissible representation. We define
$\wedge^2(\sigma)\colon W_F\to\GL(\wedge^2V)$ by
\[
\wedge^2(\sigma)(w)(v_1\wedge v_2)
=
\sigma(w)v_1\wedge\sigma(w)v_2,
\qquad w\in W_F,\quad v_1,v_2\in V.
\]
Then $\wedge^2(\sigma)$ is an admissible representation of $W_F$. We write
\[
\begin{aligned}
L(s,\sigma,\wedge^2)
&=L\bigl(s,\wedge^2(\sigma)\bigr),\\
\varepsilon(s,\sigma,\wedge^2,\psi_F)
&=\varepsilon\bigl(s,\wedge^2(\sigma),\psi_F\bigr).
\end{aligned}
\]
Suppose that
\[
\sigma=\sigma_1\oplus\cdots\oplus\sigma_k
\]
is a direct sum of admissible representations of $W_F$. Then
\[
\wedge^2(\sigma)
\simeq
\bigoplus_{i=1}^k\wedge^2(\sigma_i)
\oplus
\bigoplus_{1\leq i<j\leq k}
\sigma_i\otimes\sigma_j.
\]

Let $\pi\in\operatorname{Irr}(G_n(F))$. The local Langlands
correspondence for $G_n(F)$ \cite[Theorems~2 and~5]{Kna1994} associates with $\pi$ an
$n$-dimensional admissible representation
\[
\sigma_\pi\colon W_F\to\GL(V).
\]
We define the exterior-square \(L\)-function and \(\varepsilon\)-factor of
\(\pi\), respectively, by
\[
\begin{aligned}
L(s,\pi,\wedge^2)
&=L\bigl(s,\wedge^2(\sigma_\pi)\bigr),\\
\varepsilon(s,\pi,\wedge^2,\psi_F)
&=\varepsilon\bigl(s,\wedge^2(\sigma_\pi),\psi_F\bigr).
\end{aligned}
\]
\subsection{Families of local integrals and their functional equations}

Let \(\sigma\) be an admissible \(n\)-dimensional representation of
\(W_{\mathbb{C}}\). We define
\[
\mathcal{I}_{\As}(\sigma)
=
\operatorname{span}_{\mathbb{C}}
\left\{
I(s,W,\phi):
W\in\mathcal{W}(\pi_\sigma,\psi_{\mathbb{C}}),\
\phi\in\mathcal{S}(\mathbb{R}^n)
\right\}.
\]

The proof of \cite[Theorem~1(i),(ii)]{BP2021}, in particular the
argument in \cite[Section~3.10]{BP2021}, gives the following functional
equation for the generic induced representation \(\pi_\sigma\).

\begin{thm}
\label{fe1}
For every
\(W\in\mathcal{W}(\pi_\sigma,\psi_{\mathbb{C}})\) and
\(\phi\in\mathcal{S}(\mathbb{R}^n)\), the integral \(I(s,W,\phi)\)
admits a meromorphic continuation to \(\mathbb{C}\) and satisfies
\[
\frac{I(1-s,\widetilde W,\widehat\phi)}
     {L\bigl(1-s,\As(\sigma)^\vee\bigr)}
=
\omega_{\pi_\sigma}(\tau)^{n-1}
|\tau|_{\mathbb{C}}^{\frac{n(n-1)}{2}(s-\frac12)}
\lambda_{\mathbb{C}/\mathbb{R}}(\psi_{\mathbb{R}})
^{-\frac{n(n-1)}{2}}
\varepsilon\bigl(s,\As(\sigma),\psi_{\mathbb{R}}\bigr)
\frac{I(s,W,\phi)}
     {L\bigl(s,\As(\sigma)\bigr)}.
\]
\end{thm}

The factor preceding the epsilon factor in Theorem~\ref{fe1} is of the
form \(cA^s\), where \(c\in\mathbb{C}^{\times}\) and \(A>0\). Since
\(\mathcal I_{\As}(\sigma)\) is invariant under multiplication by
such functions, and the maps
$W\longmapsto\widetilde W
\>\text{and}\>
\phi\longmapsto\widehat\phi
$ are bijective, Theorem~\ref{fe1} gives
\begin{equation}\label{flickerfe1}
\mathcal{I}_{\As}(\sigma)
=
\gamma\bigl(s,\As(\sigma),\psi_{\mathbb{R}}\bigr)^{-1}
\left\{
f(1-s):
f\in\mathcal{I}_{\As}(\sigma^\vee)
\right\}.
\end{equation}
In \eqref{flickerfe1}, the space
\(\mathcal{I}_{\As}(\sigma^\vee)\) is defined using the Whittaker realization
\(\mathcal{W}(\pi_{\sigma^\vee},\psi_{\mathbb{C}}^{-1})\).

Let \(\sigma\) be an admissible \(m\)-dimensional representation of
\(W_F\). We define
\[
\mathcal{J}_{\wedge^2}(\sigma)
=
\begin{cases}
\displaystyle
\operatorname{span}_{\mathbb{C}}
\left\{
J(s,W,\phi):
W\in\mathcal{W}(\pi_\sigma,\psi_F),\
\phi\in\mathcal{S}(F^n)
\right\},
& m=2n,\\[4mm]
\displaystyle
\operatorname{span}_{\mathbb{C}}
\left\{
J(s,W):
W\in\mathcal{W}(\pi_\sigma,\psi_F)
\right\},
& m=2n+1,
\end{cases}
\]
where the Jacquet--Shalika integrals are those defined in
Section~\ref{s1}.

When \(m=2n+1\), for
\(W\in\mathcal{W}(\pi_\sigma,\psi_F)\) and
\(\phi\in\mathcal{S}(F^n)\), set
\[
R(\phi)W
=
\int_{F^n}
\phi(x)\,
R\left(
\begin{pmatrix}
I_n&0&0\\
0&I_n&0\\
0&x&1
\end{pmatrix}
\right)W\,dx.
\]
With this notation, the odd Jacquet--Shalika integral in
\cite[(2.9), (3.4)]{JLST2025} is \(J(s,R(\phi)W)\). By the
Dixmier--Malliavin lemma \cite[Lemma~6.1]{Jac2009},
\[
\mathcal{W}(\pi_\sigma,\psi_F)
=
\operatorname{span}_{\mathbb{C}}
\left\{
R(\phi)W:
W\in\mathcal{W}(\pi_\sigma,\psi_F),\
\phi\in\mathcal{S}(F^n)
\right\}.
\]
It follows that
\[
\mathcal{J}_{\wedge^2}(\sigma)
=
\operatorname{span}_{\mathbb{C}}
\left\{
J(s,R(\phi)W):
W\in\mathcal{W}(\pi_\sigma,\psi_F),\
\phi\in\mathcal{S}(F^n)
\right\}.
\]

Together with this observation,
\cite[Theorem~2.2(1),(2)]{JLST2025} gives the following.

\begin{thm}
\label{fe2}
All the Jacquet--Shalika integrals defining
\(\mathcal{J}_{\wedge^2}(\sigma)\) admit meromorphic continuation to
\(\mathbb{C}\).

If \(m=2n\), then, for every
\(W\in\mathcal{W}(\pi_\sigma,\psi_F)\) and
\(\phi\in\mathcal{S}(F^n)\),
\[
\frac{
J\bigl(
1-s,
R(\tau_{2n})\widetilde W,
\mathcal{F}_{\psi_F}(\phi)
\bigr)
}{
L\bigl(1-s,(\wedge^2\sigma)^\vee\bigr)
}
=
\varepsilon\bigl(s,\wedge^2\sigma,\psi_F\bigr)
\frac{J(s,W,\phi)}
     {L\bigl(s,\wedge^2\sigma\bigr)}.
\]

If \(m=2n+1\), then, for every
\(W\in\mathcal{W}(\pi_\sigma,\psi_F)\) and
\(\phi\in\mathcal{S}(F^n)\),
\[
\frac{
J\left(
1-s,
R\bigl(\mathcal{F}_{\psi_F^{-1}}(\phi)\bigr)
\bigl(R(\tau_{2n+1})\widetilde W\bigr)
\right)
}{
L\bigl(1-s,(\wedge^2\sigma)^\vee\bigr)
}
=
\varepsilon\bigl(s,\wedge^2\sigma,\psi_F\bigr)
\frac{J\bigl(s,R(\phi)W\bigr)}
     {L\bigl(s,\wedge^2\sigma\bigr)}.
\]
\end{thm}

The Fourier transforms occurring in Theorem~\ref{fe2} are bijective,
as is the map
\(W\longmapsto R(\tau_m)\widetilde W\). Moreover, in odd rank, the
vectors \(R(\phi)W\) span the Whittaker model. Thus, in both parities,
Theorem~\ref{fe2} yields
\begin{equation}
\label{shalikafe2}
\mathcal{J}_{\wedge^2}(\sigma)
=
\gamma\bigl(s,\wedge^2\sigma,\psi_F\bigr)^{-1}
\left\{
f(1-s):
f\in\mathcal{J}_{\wedge^2}(\sigma^\vee)
\right\}.
\end{equation}
In \eqref{shalikafe2}, the space
\(\mathcal{J}_{\wedge^2}(\sigma^\vee)\) is defined using
\(\mathcal{W}(\pi_{\sigma^\vee},\psi_F^{-1})\) and the additive
character \(\psi_F^{-1}\) in the Jacquet--Shalika integrals.

\begin{section}{\texorpdfstring{$L$}{L}-Spaces}\label{s3}

We first recall the notions of rapid decrease and boundedness at infinity in vertical strips, and then define the \(L\)-space associated with a Weil-group representation.

\begin{defn}
A meromorphic function $f$ on $\mathbb{C}$ is said to be
\emph{rapidly decreasing in vertical strips} if, for every
$N\in\mathbb{N}\cup\{0\}$ and every vertical strip
\[
a\leq\operatorname{Re}(s)\leq b,
\]
the function $s^Nf(s)$ is bounded as
$|\operatorname{Im}(s)|\to\infty$, uniformly in the strip. If this
condition is required only for $N=0$, we say that $f$ is
\emph{bounded at infinity in vertical strips}.
\end{defn}

\begin{defn}
Let $\sigma$ be an admissible representation of $W_F$. The
\emph{$L$-space of $\sigma$}, denoted by $\mathcal{L}(\sigma)$, is the
space of meromorphic functions $f$ for which there exists an entire
function $h$ such that
\[
f(s)=L(s,\sigma)h(s)
\]
and $f$ is rapidly decreasing in vertical strips.
\end{defn}

In particular, $\mathcal{L}(0)$ denotes the space of entire functions rapidly decreasing in vertical strips. By Stirling's formula, we have $L(s,\sigma)\in\mathcal{L}(\sigma)$ for every admissible representation $\sigma$ of $W_F$.

In this section, we establish a decomposition identity for the
\(L\)-space attached to a representation with two direct-sum
decompositions. This identity will be used in the inductive proofs of
Theorems~\ref{main1} and~\ref{main2}. We begin by recalling three
results of Jacquet \cite{Jac2009}.

\begin{lemma}[{\cite[Lemma~12.1]{Jac2009}}]
\label{lem:Jac2009-12.1}
Let $\sigma_1$ be a subrepresentation of an admissible representation
$\sigma$ of $W_F$. Then $\mathcal{L}(\sigma_1)\subseteq\mathcal{L}(\sigma).$ 
In particular, $\mathcal{L}(0)\subseteq\mathcal{L}(\sigma).
$
\end{lemma}

\begin{prop}[{\cite[Proposition~12.1]{Jac2009}}]
\label{prop:Jac2009-12.1}
For each $s_0\in\mathcal{P}(L(s,\sigma))$, let $n_{s_0}$ denote the
order of the pole at $s_0$. Suppose that, for every
$s_0\in\mathcal{P}(L(s,\sigma))$, a principal part
\[
\mathcal{Q}(s_0)
=
\frac{A_{n_{s_0}}}{(s-s_0)^{n_{s_0}}}
+
\frac{A_{n_{s_0}-1}}{(s-s_0)^{n_{s_0}-1}}
+\cdots+
\frac{A_1}{s-s_0}
\]
is prescribed. Then there exists $f\in\mathcal{L}(\sigma)$ whose
principal part at each $s_0\in\mathcal{P}(L(s,\sigma))$ is
$\mathcal{Q}(s_0)$.
\end{prop}

\begin{lemma}[{\cite[Lemma~12.2]{Jac2009}}]
\label{lem:Jac2009-12.2}
For every admissible representation $\sigma$ of $W_F$, we have
\[
\mathcal{L}(\sigma)
=
\mathcal{L}(0)L(s,\sigma)+\mathcal{L}(0).
\]
\end{lemma}

The preceding results yield the following decomposition identity.

\begin{prop}[\(L\)-space decomposition] \label{prop:two-decomposition} Let $\theta,\rho_1,\rho_2,\eta_1,$ and $\eta_2$ be admissible representations of $W_F$ such that \[ \theta = \rho_1\oplus\rho_2 = \eta_1\oplus\eta_2. \] Assume that \begin{equation} \label{eq:pole-separation} \mathcal{P}\bigl(L(1-s,\rho_1^\vee)\bigr) \cap \mathcal{P}\bigl(L(s,\eta_2)\bigr) = \varnothing. \end{equation} Then \begin{equation} \label{eq:two-decomposition} \mathcal{L}(\theta) = \mathcal{L}(\eta_1) + \frac{L(s,\rho_1)}{L(1-s,\rho_1^\vee)} \mathcal{L}(\rho_2). \end{equation} \end{prop}

\begin{proof}
We first prove that the right-hand side of
\eqref{eq:two-decomposition} is contained in $\mathcal{L}(\theta)$.
Since $\eta_1$ is a direct summand of $\theta$, Lemma
\ref{lem:Jac2009-12.1} gives
\[
\mathcal{L}(\eta_1)\subseteq\mathcal{L}(\theta).
\]
Let $f(s)=L(s,\rho_2)h(s)\in\mathcal{L}(\rho_2)$, where $h$ is entire. Then
\[
\frac{L(s,\rho_1)}{L(1-s,\rho_1^\vee)}f(s)
=
L(s,\theta)\frac{h(s)}{L(1-s,\rho_1^\vee)}.
\]
Since \(1/L(1-s,\rho_1^\vee)\) is entire, the function
$h(s)/L(1-s,\rho_1^\vee)$ is entire. Moreover, Stirling's formula
shows that
\[
\frac{L(s,\rho_1)}{L(1-s,\rho_1^\vee)}
\]
has at most polynomial growth at infinity in vertical strips. It
follows that
\[
\frac{L(s,\rho_1)}{L(1-s,\rho_1^\vee)}f(s)
\in\mathcal{L}(\theta).
\]

We now prove the reverse inclusion. Set
\[
d(s)=L(1-s,\rho_1^\vee),
\qquad
m(s)=\frac{L(s,\rho_1)}{d(s)},
\qquad
\Omega=\mathcal{P}(d).
\]
Let $f\in\mathcal{L}(\theta)$ and write
\[
f(s)=L(s,\rho_1)L(s,\rho_2)h(s),
\]
where $h$ is entire. Define
\[
G(s)=\frac{d(s)f(s)}{L(s,\rho_1)}
=
d(s)L(s,\rho_2)h(s).
\]
Outside $\Omega$, the function $d$ is holomorphic and nonvanishing.
Thus every pole of $G$ outside $\Omega$ is a pole of
$L(s,\rho_2)$, with order no greater than the corresponding pole
order of $L(s,\rho_2)$. By Proposition
\ref{prop:Jac2009-12.1}, there exists
$g_1\in\mathcal{L}(\rho_2)$ whose principal parts agree with those of
$G$ at the poles of $L(s,\rho_2)$ outside $\Omega$ and vanish at the
poles lying in $\Omega$.

Set
\[
H_1=f-mg_1.
\]
Then
\[
\frac{d(s)H_1(s)}{L(s,\rho_1)}
=
G(s)-g_1(s)
\]
is holomorphic outside $\Omega$. Hence every pole of $dH_1$ outside
$\Omega$ is a pole of $L(s,\rho_1)$, with order no greater than the
corresponding pole order of $L(s,\rho_1)$. Applying Proposition
\ref{prop:Jac2009-12.1} again, we obtain
$K\in\mathcal{L}(\rho_1)$ whose principal parts agree with those of
$dH_1$ at the poles of $L(s,\rho_1)$ outside $\Omega$ and vanish at
the poles lying in $\Omega$.

By Lemma \ref{lem:Jac2009-12.2}, we may write
\[
K(s)=L(s,\rho_1)g_0(s)+e_0(s),
\qquad
g_0,e_0\in\mathcal{L}(0).
\]
Since $e_0$ is entire, the functions $K$ and
$L(s,\rho_1)g_0$ have the same principal parts. By the choice of
$K$, the function
\[
d(s)H_1(s)-K(s)
\]
is holomorphic outside $\Omega$.

By Lemma \ref{lem:Jac2009-12.1}, we have
$\mathcal{L}(0)\subseteq\mathcal{L}(\rho_2)$. Therefore, $g=g_1+g_0\in\mathcal{L}(\rho_2)$. Set
\[
H=f-mg=H_1-mg_0.
\]
Using the expression for $K$, we obtain
\[
\begin{aligned}
d(s)H(s)
&=
d(s)H_1(s)-L(s,\rho_1)g_0(s)\\
&=
\bigl(d(s)H_1(s)-K(s)\bigr)+e_0(s).
\end{aligned}
\]
Thus $dH$ is holomorphic outside $\Omega$. Since $d$ is holomorphic
and nonvanishing there, $H$ is holomorphic outside $\Omega$.

By the first inclusion, $mg\in\mathcal{L}(\theta)$, and hence
$H=f-mg\in\mathcal{L}(\theta)$. We may therefore write
\[
H(s)=L(s,\eta_1)L(s,\eta_2)q(s)
\]
for some entire function $q$. Since \(1/L(s,\eta_1)\) is entire and \(H\) is holomorphic outside
\(\Omega\), the quotient \(H/L(s,\eta_1)\) is holomorphic outside
\(\Omega\). If $s_0\in\Omega$, then
\eqref{eq:pole-separation} implies that $L(s,\eta_2)$ is holomorphic
at $s_0$, and hence
\[
\frac{H(s)}{L(s,\eta_1)}
=
L(s,\eta_2)q(s)
\]
is holomorphic at $s_0$. Thus $H/L(s,\eta_1)$ is entire.

Finally, $H$ is rapidly decreasing in vertical strips, since both
$f$ and $mg$ belong to $\mathcal{L}(\theta)$. Consequently,
\[
H\in\mathcal{L}(\eta_1).
\]
Therefore,
\[
f
=
H+mg
\in
\mathcal{L}(\eta_1)
+
\frac{L(s,\rho_1)}{L(1-s,\rho_1^\vee)}
\mathcal{L}(\rho_2),
\]
which proves \eqref{eq:two-decomposition}.
\end{proof}
    
\end{section}

\begin{section}{Proof of Theorem 1.1}\label{s4}

Let $\sigma$ be an $n$-dimensional admissible representation of $W_{\mathbb{C}}$. Since $W_{\mathbb{C}}=\mathbb{C}^{\times}$ is abelian, every irreducible constituent of $\sigma$ is a character. Ordering these by their exponents $e_i$, we write
\[
  \sigma = \chi_1 \oplus \chi_2 \oplus \cdots \oplus \chi_n, \qquad e_1 \leq e_2 \leq \cdots \leq e_n.
\]
Define the truncated representations $\sigma^+ = \chi_2\oplus\cdots\oplus\chi_n$ and $\sigma^- = \chi_1\oplus\cdots\oplus\chi_{n-1}$.

Applying the direct-sum formula for the Asai representation by isolating $\chi_1$ gives $\As(\sigma) = \rho_1 \oplus \rho_2$, where
\[
  \rho_1 = \As(\chi_1) \oplus \operatorname{Ind}_{W_{\mathbb{C}}}^{W_{\mathbb{R}}} \bigl(\chi_1\otimes(\sigma^+)^c\bigr) \qquad\text{and}\qquad \rho_2 = \As(\sigma^+).
\]

Alternatively, applying the direct-sum formula by isolating $\chi_n$ yields $\As(\sigma) = \eta_1 \oplus \eta_2$, where
\[
  \eta_1 = \As(\sigma^-) \qquad\text{and}\qquad \eta_2 = \As(\chi_n) \oplus \operatorname{Ind}_{W_{\mathbb{C}}}^{W_{\mathbb{R}}} \bigl(\sigma^-\otimes\chi_n^c\bigr).
\]

Thus, we obtain the two decompositions:
\[
  \As(\sigma) = \rho_1 \oplus \rho_2 = \eta_1 \oplus \eta_2.
\]

\begin{lemma}
\label{lem:asai-pole-separation}
With the notation above,
\[
  \mathcal{P}\bigl(L(1-s,\rho_1^\vee)\bigr) \cap \mathcal{P}\bigl(L(s,\eta_2)\bigr) = \varnothing.
\]
\end{lemma}

\begin{proof}
Suppose that $s_0$ is a common pole. The irreducible constituents of $\rho_1$ have exponents $e_1+e_j$ for $1\leq j\leq n$ (where $j=1$ corresponds to $\As(\chi_1)$). Therefore, there exists some $j\in\{1,\ldots,n\}$ such that
\[
  1-\operatorname{Re}(s_0)-e_1-e_j \leq 0.
\]

Similarly, the irreducible constituents of $\eta_2$ have exponents $e_i+e_n$ for $1\leq i\leq n$ (where $i=n$ corresponds to $\As(\chi_n)$), meaning there exists some $i\in\{1,\ldots,n\}$ such that
\[
  \operatorname{Re}(s_0)+e_i+e_n \leq 0.
\]

Adding these two inequalities yields
\[
  1+(e_i-e_1)+(e_n-e_j) \leq 0.
\]
This is a contradiction since the ordering of the exponents implies $e_i\geq e_1$ and $e_n\geq e_j$.
\end{proof}

Using Lemma~\ref{lem:asai-pole-separation} and Proposition~\ref{prop:two-decomposition}, we obtain the following.

\begin{prop}[Asai \(L\)-space decomposition]
\label{prop:asai-L-space-decomposition}
With the notation above, one has
\[
\mathcal{L}\bigl(\As(\sigma)\bigr)
=
\mathcal{L}\bigl(\As(\sigma^-)\bigr)
+
\frac{L(s,\rho_1)}{L(1-s,\rho_1^\vee)}
\mathcal{L}\bigl(\As(\sigma^+)\bigr).
\]
Equivalently,
\[
\mathcal{L}\bigl(\As(\sigma)\bigr)
=
\mathcal{L}\bigl(\As(\sigma^-)\bigr)
+
\frac{
L\bigl(s,\As(\chi_1)\bigr)
L\left(
s,
\operatorname{Ind}_{W_{\mathbb{C}}}^{W_{\mathbb{R}}}
\bigl(\chi_1\otimes(\sigma^+)^c\bigr)
\right)
}{
L\bigl(1-s,\As(\chi_1^\vee)\bigr)
L\left(
1-s,
\operatorname{Ind}_{W_{\mathbb{C}}}^{W_{\mathbb{R}}}
\bigl(\chi_1^\vee\otimes((\sigma^+)^\vee)^c\bigr)
\right)
}
\mathcal{L}\bigl(\As(\sigma^+)\bigr).
\]
\end{prop}

We next turn to the inductive reduction of Flicker integrals. Let $\xi$
and $\zeta$ be admissible representations of $W_{\mathbb{C}}$ such that  $\xi\preceq\zeta$.
Set
\[
n_\xi=\dim(\xi)
\qquad\text{and}\qquad
n_\zeta=\dim(\zeta).
\]
The following lemma is the Flicker analogue of
\cite[Proposition~14.1]{Jac2009}.

\begin{lemma}
\label{lem:whittaker-replacement}
Given
$W_\xi\in\mathcal{W}(\pi_\xi,\psi_{\mathbb{C}})$ and
$\phi\in\mathcal{S}(\mathbb{R}^{n_\xi})$, there exists
$W\in\mathcal{W}(\pi_{\xi\oplus\zeta},\psi_{\mathbb{C}})$ such that,
for every $g\in G_{n_\xi}(\mathbb{R})$,
\[
W
\begin{pmatrix}
g&0\\
0&I_{n_\zeta}
\end{pmatrix}
=
W_\xi(g)\phi(e_{n_\xi}g)
|\det g|_{\mathbb{R}}^{n_\zeta}.
\]
\end{lemma}

\begin{proof}
Choose $\phi_\mathbb{C}\in \mathcal{S}(\mathbb{C}^{n_\xi})$ whose restriction to
$\mathbb{R}^{n_\xi}$ is $\phi$; for instance, set
\[
\phi_\mathbb{C}(x+iy)=\phi(x)e^{-\lVert y\rVert^2}.
\]
Applying \cite[Proposition~14.1]{Jac2009} over $\mathbb{C}$ to $W_\xi$ and $\phi_\mathbb{C}$, we obtain
$W\in\mathcal{W}(\pi_{\xi\oplus\zeta},\psi_{\mathbb{C}})$ such that
\[
W\begin{pmatrix}
g&0\\
0&I_{n_\zeta}
\end{pmatrix}
=
W_\xi(g)\phi_\mathbb{C}(e_{n_\xi}g)
\lvert\det g\rvert_{\mathbb{C}}^{n_\zeta/2}
\]
for every $g\in G_{n_\xi}(\mathbb{C})$. Restricting to
$g\in G_{n_\xi}(\mathbb{R})$, we have
\[
\phi_\mathbb{C}(e_{n_\xi}g)=\phi(e_{n_\xi}g)
\quad\text{and}\quad
\lvert\det g\rvert_{\mathbb{C}}^{n_\zeta/2}
=
\lvert\det g\rvert_{\mathbb{R}}^{n_\zeta},
\]
which proves the lemma.
\end{proof}

\begin{lemma}
\label{lem:partial-flicker-integrals}
Let $\sigma$ be an admissible representation of $W_{\mathbb C}$, and set
$n=\dim(\sigma)$. For $1\leq r<n$ and
$W\in\mathcal{W}(\pi_\sigma,\psi_{\mathbb C})$, set
\[
I_r(s,W)
=
\int_{N_r(\mathbb R)\backslash G_r(\mathbb R)}
W
\begin{pmatrix}
g&0\\
0&I_{n-r}
\end{pmatrix}
|\det g|_{\mathbb R}^{s-(n-r)}\,dg.
\]
Then $I_r(s,W)$ belongs to $\mathcal{I}_{\As}(\sigma)$. 
\end{lemma}

\begin{proof}
Suppose first that $r<n-1$, and set
\[
a_r(h)=\diag(h,I_{r-1}),
\qquad h\in\mathbb R^\times.
\]
Applying the Dixmier--Malliavin lemma \cite[Lemma~6.1]{Jac2009} to the restriction of the
representation to the subgroup consisting of the matrices
\[
\begin{pmatrix}
a_r(h)^{-1}&0&0\\
x&h&0\\
0&0&I_{n-r-1}
\end{pmatrix},
\qquad
x\in\mathbb R^r,\quad h\in\mathbb R^\times,
\]
we may write
\[
W
=
\sum_{i=1}^q
\int_{\mathbb R^r\times\mathbb R^\times}
R\left(
\begin{pmatrix}
a_r(h)^{-1}&0&0\\
x&h&0\\
0&0&I_{n-r-1}
\end{pmatrix}
\right)
W_i\,
\varphi_i(x,h)\,dx\,|h|_{\mathbb R}\,d^\times h,
\]
where $W_i\in\mathcal{W}(\pi_\sigma,\psi_{\mathbb C})$ and 
$\varphi_i\in C_c^\infty(\mathbb R^r\times\mathbb R^\times).$
Substituting this expression into $I_r(s,W)$ and making the change of
variables $g a_r(h)^{-1}\mapsto g$, we obtain
\[
\begin{aligned}
I_r(s,W)
={}&
\sum_{i=1}^q
\int_{\mathbb R^\times}
\int_{\mathbb R^r}
\int_{N_r(\mathbb R)\backslash G_r(\mathbb R)}
W_i
\begin{pmatrix}
g&0&0\\
x&h&0\\
0&0&I_{n-r-1}
\end{pmatrix}
\varphi_i(x,h)                                     \\
&\hspace{25mm}\times
|\det g|_{\mathbb R}^{s-(n-r)}
|h|_{\mathbb R}^{s-(n-r-1)}
\,dg\,dx\,d^\times h.
\end{aligned}
\]

Extend $\varphi_i$ by zero to a function
$\varphi_i^0\in C_c^\infty(\mathbb R^{r+1})$, and define
\[
\varphi_{i,\mathbb C}(u+iv)
=
\varphi_i^0(u)e^{-\lVert v\rVert^2},
\qquad
u,v\in\mathbb R^{r+1}.
\]
Then $\varphi_{i,\mathbb C}\in\mathcal{S}(\mathbb C^{r+1})$.
By \cite[Proposition~6.1]{Jac2009}, applied over $\mathbb C$, there exists
$W_i'\in\mathcal{W}(\pi_\sigma,\psi_{\mathbb C})$ such that
\[
 W_i'
\begin{pmatrix}
\gamma&0\\
0&I_{n-r-1}
\end{pmatrix}
=
W_i
\begin{pmatrix}
\gamma&0\\
0&I_{n-r-1}
\end{pmatrix}
\varphi_{i,\mathbb C}(e_{r+1}\gamma)
\]
for every $\gamma\in G_{r+1}(\mathbb C)$. In particular, for
$\gamma=\begin{psmallmatrix}g&0\\ x&h\end{psmallmatrix}\in
G_{r+1}(\mathbb R)$, this gives
\[
W_i'
\begin{pmatrix}
g&0&0\\
x&h&0\\
0&0&I_{n-r-1}
\end{pmatrix}
=
W_i
\begin{pmatrix}
g&0&0\\
x&h&0\\
0&0&I_{n-r-1}
\end{pmatrix}
\varphi_i(x,h).
\]

With compatible Haar measures, one has
\[
\int_{N_{r+1}(\mathbb R)\backslash G_{r+1}(\mathbb R)}
F(\gamma)\,d\gamma
=
\int_{\mathbb R^\times}
\int_{\mathbb R^r}
\int_{N_r(\mathbb R)\backslash G_r(\mathbb R)}
F\begin{pmatrix}g&0\\x&h\end{pmatrix}
|\det g|_{\mathbb R}^{-1}\,dg\,dx\,d^\times h.
\]
Since
\[
\left|
\det
\begin{pmatrix}
g&0\\
x&h
\end{pmatrix}
\right|_{\mathbb R}^{s-(n-r-1)}
|\det g|_{\mathbb R}^{-1}
=
|\det g|_{\mathbb R}^{s-(n-r)}
|h|_{\mathbb R}^{s-(n-r-1)},
\]
it follows that
\begin{equation}\label{induct}
    I_r(s,W)=\sum_{i=1}^q I_{r+1}(s,W_i').
\end{equation}

It remains to treat the final step. Applying the same
Dixmier--Malliavin argument when $r=n-1$, we obtain
\[
\begin{aligned}
I_{n-1}(s,W)
={}&
\sum_{i=1}^q
\int_{\mathbb R^\times}
\int_{\mathbb R^{n-1}}
\int_{N_{n-1}(\mathbb R)\backslash G_{n-1}(\mathbb R)}
W_i
\begin{pmatrix}
g&0\\
x&h
\end{pmatrix}                                      \\
&\hspace{25mm}\times
\varphi_i(x,h)
|\det g|_{\mathbb R}^{s-1}
|h|_{\mathbb R}^{s}
\,dg\,dx\,d^\times h.
\end{aligned}
\]
Define $\phi_i\in C_c^\infty(\mathbb R^n)$ by
\[
\phi_i(x,h)
=
\begin{cases}
\varphi_i(x,h),&h\neq0,\\
0,&h=0.
\end{cases}
\]
This is smooth because $\varphi_i$ has compact support in
$\mathbb R^{n-1}\times\mathbb R^\times$. The preceding measure
decomposition now gives
\[
I_{n-1}(s,W)
=
\sum_{i=1}^q
\int_{N_n(\mathbb R)\backslash G_n(\mathbb R)}
W_i(g)\phi_i(e_ng)|\det g|_{\mathbb R}^{s}\,dg
=
\sum_{i=1}^q I(s,W_i,\phi_i).
\]
Iterating \eqref{induct} completes the proof.
\end{proof}

Using Lemma~\ref{lem:whittaker-replacement},  Lemma~\ref{lem:partial-flicker-integrals} and the functional equation
for the Flicker integrals (Theorem \ref{fe1}), we prove the following.

\begin{prop}[Flicker integral reduction]
\label{prop:asai-integral-inclusions}
Let $\xi$ and $\zeta$ be admissible representations of
$W_{\mathbb{C}}$ such that $\xi\preceq\zeta$. Then
\[
\mathcal{I}_\As\bigl(\xi\oplus\zeta\bigr)
\supseteq
\mathcal{I}_\As\bigl(\xi\bigr)
\]
and
\[
\mathcal{I}_\As\bigl(\xi\oplus\zeta\bigr)
\supseteq
R_{\xi,\zeta}(s)\,
\mathcal{I}_\As\bigl(\zeta\bigr),
\]
where
\[
R_{\xi,\zeta}(s)
=
\frac{
L\bigl(s,\As(\xi)\bigr)
L\left(
s,
\operatorname{Ind}_{W_{\mathbb{C}}}^{W_{\mathbb{R}}}
\bigl(\xi\otimes\zeta^c\bigr)
\right)
}{
L\bigl(1-s,\As(\xi^\vee)\bigr)
L\left(
1-s,
\operatorname{Ind}_{W_{\mathbb{C}}}^{W_{\mathbb{R}}}
\bigl(\xi^\vee\otimes(\zeta^\vee)^c\bigr)
\right)
}.
\]
\end{prop}

\begin{proof}
Let \(I(s,W_\xi,\phi)\in\mathcal{I}_\As(\xi)\), where
\(W_\xi\in\mathcal{W}(\pi_\xi,\psi_{\mathbb C})\) and
\(\phi\in\mathcal{S}(\mathbb R^{n_\xi})\), be given by
\[
I(s,W_\xi,\phi)
=
\int_{N_{n_\xi}(\mathbb R)\backslash G_{n_\xi}(\mathbb R)}
W_\xi(g)\phi(e_{n_\xi}g)|\det g|_{\mathbb R}^s\,dg.
\]

By Lemma~\ref{lem:whittaker-replacement}, there exists
$W\in\mathcal{W}(\pi_{\xi\oplus\zeta},\psi_{\mathbb{C}})$ such that
\[
W
\begin{pmatrix}
g&0\\
0&I_{n_\zeta}
\end{pmatrix}
=
W_\xi(g)\phi(e_{n_\xi}g)
|\det g|_{\mathbb{R}}^{n_\zeta}
\]
for every $g\in G_{n_\xi}(\mathbb{R})$. Therefore,
\[
I(s,W_\xi,\phi)
=
\int_{N_{n_\xi}(\mathbb{R})\backslash
G_{n_\xi}(\mathbb{R})}
W
\begin{pmatrix}
g&0\\
0&I_{n_\zeta}
\end{pmatrix}
|\det g|_{\mathbb{R}}^{s-n_\zeta}\,dg.
\]

By Lemma \ref{lem:partial-flicker-integrals}, there exists $r\geq 1$,
$W_i\in\mathcal{W}(\pi_{\xi\oplus\zeta},\psi_{\mathbb{C}})$ and
$\phi_i\in\mathcal{S}(\mathbb{R}^{n_\xi+n_\zeta})$ such that
\[
I(s,W_\xi,\phi)
=
\sum_{i=1}^r I(s,W_i,\phi_i).
\]
Consequently,
\[
\mathcal{I}_\As\bigl(\xi\oplus\zeta\bigr)
\supseteq
\mathcal{I}_\As\bigl(\xi\bigr).
\]

For the second inclusion, note that
\[
\zeta^\vee\preceq\xi^\vee.
\]
Applying the first inclusion to
\[
(\xi\oplus\zeta)^\vee
=
\zeta^\vee\oplus\xi^\vee
\]
gives
\[
\mathcal{I}_\As\bigl((\xi\oplus\zeta)^\vee\bigr)
\supseteq
\mathcal{I}_\As\bigl(\zeta^\vee\bigr).
\]

The direct-sum formula for the Asai representation gives
\[
\As(\xi\oplus\zeta)
\simeq
\As(\xi)
\oplus
\As(\zeta)
\oplus
\operatorname{Ind}_{W_{\mathbb{C}}}^{W_{\mathbb{R}}}
\bigl(\xi\otimes\zeta^c\bigr).
\]
Hence, by \eqref{flickerfe1},
$\mathcal{I}_\As(\xi\oplus\zeta)$ contains
\[
\gamma\bigl(s,\As(\xi),\psi_{\mathbb{R}}\bigr)^{-1}
\gamma\left(
s,
\operatorname{Ind}_{W_{\mathbb{C}}}^{W_{\mathbb{R}}}
\bigl(\xi\otimes\zeta^c\bigr),
\psi_{\mathbb{R}}
\right)^{-1}
\mathcal{I}_\As\bigl(\zeta\bigr).
\]
Using
\[
\gamma(s,\rho,\psi_{\mathbb{R}})
=
\varepsilon(s,\rho,\psi_{\mathbb{R}})
\frac{L(1-s,\rho^\vee)}{L(s,\rho)}
\]
and the fact that $\mathcal{I}_\As\bigl(\zeta\bigr)$ is invariant under
multiplication by functions of the form \(cA^s\), where
\(c\in\mathbb{C}^{\times}\) and \(A>0\), we obtain
\[
\mathcal{I}_\As\bigl(\xi\oplus\zeta\bigr)
\supseteq
\frac{
L\bigl(s,\As(\xi)\bigr)
L\left(
s,
\operatorname{Ind}_{W_{\mathbb{C}}}^{W_{\mathbb{R}}}
\bigl(\xi\otimes\zeta^c\bigr)
\right)
}{
L\bigl(1-s,\As(\xi^\vee)\bigr)
L\left(
1-s,
\operatorname{Ind}_{W_{\mathbb{C}}}^{W_{\mathbb{R}}}
\bigl(\xi^\vee\otimes(\zeta^\vee)^c\bigr)
\right)
}
\mathcal{I}_\As\bigl(\zeta\bigr).
\]
This proves the second inclusion.
\end{proof}

To prove Theorem~\ref{main1}, it suffices to show that
\[
\mathcal{I}_\As\bigl(\sigma\bigr)
\supseteq
\mathcal{L}\bigl(\As(\sigma)\bigr).
\]
We first establish this inclusion in the case when $\sigma$ is one-dimensional.

\begin{prop}
\label{prop:asai-rank-one}
Let $\sigma$ be a one-dimensional admissible representation of
$W_{\mathbb C}$. Then
\[
\mathcal{I}_\As(\sigma)
\supseteq
\mathcal{L}(\As(\sigma)).
\]
\end{prop}

\begin{proof}
Write
\[
\sigma(z)
=
|z|_{\mathbb C}^{t}
\left(
\frac{z}{|z|_{\mathbb C}^{1/2}}
\right)^{m},
\] where \(t\in\mathbb C\) and \(m\in\mathbb Z\). Let
\(\varepsilon\in\{0,1\}\) be determined by
\(\varepsilon\equiv m\pmod 2\), and set
\(\omega=\operatorname{sgn}^{\varepsilon}\). Then \(\omega\) is
normalized,
\[
\As(\sigma)=\omega|\cdot|_{\mathbb R}^{2t},
\qquad
L(s,\As(\sigma))=L(s+2t,\omega).
\]

Let \(F\in\mathcal{L}(\As(\sigma))\), and set
\(F_0(s)=F(s-2t)\). Translation in \(s\) preserves the defining
vertical-strip estimates, so \(F_0\in\mathcal{L}(\omega)\). By
\cite[Lemma~12.3]{Jac2009}, there exists
\(\phi\in\mathcal{S}(\mathbb R)\) such that
\[
F_0(s)
=
\int_{\mathbb R^\times}
\phi(x)\omega(x)|x|_{\mathbb R}^{s}\,d^\times x.
\]
Replacing \(s\) by \(s+2t\), and using
\(W_\sigma(x)=\sigma(x)=\omega(x)|x|_{\mathbb R}^{2t}\) for
\(x\in\mathbb R^\times\), gives
\[
F(s)
=
\int_{\mathbb R^\times}
W_\sigma(x)\phi(x)|x|_{\mathbb R}^{s}\,d^\times x
=
I(s,W_\sigma,\phi).
\]
Hence \(F\in\mathcal{I}_{\As}(\sigma)\).
\end{proof}

We now prove the general case by induction.

\begin{proof}[Proof of Theorem~\ref{main1}]
We argue by induction on $n=\dim(\sigma)$. The case $n=1$ follows from
Proposition~\ref{prop:asai-rank-one}. Assume that $n>1$ and that the
result holds for all admissible representations of $W_{\mathbb{C}}$ of
dimension strictly less than $n$.

Using the notation introduced at the beginning of this section, write
\[
\sigma
=
\sigma^-\oplus\chi_n
=
\chi_1\oplus\sigma^+.
\]
Since $\sigma^-\preceq\chi_n$, Proposition
\ref{prop:asai-integral-inclusions} gives
\[
\mathcal{I}_\As\bigl(\sigma\bigr)
\supseteq
\mathcal{I}_\As\bigl(\sigma^-\bigr).
\]
By the induction hypothesis,
\[
\mathcal{I}_\As\bigl(\sigma^-\bigr)
\supseteq
\mathcal{L}\bigl(\As(\sigma^-)\bigr),
\]
and hence
\[
\mathcal{I}_\As\bigl(\sigma\bigr)
\supseteq
\mathcal{L}\bigl(\As(\sigma^-)\bigr).
\]

Similarly, since $\chi_1\preceq\sigma^+$, Proposition
\ref{prop:asai-integral-inclusions} and the induction hypothesis give
\[
\mathcal{I}_\As\bigl(\sigma\bigr)
\supseteq
R_{\chi_1,\sigma^+}(s)\,
\mathcal{I}_\As\bigl(\sigma^+\bigr)
\supseteq
R_{\chi_1,\sigma^+}(s)\,
\mathcal{L}\bigl(\As(\sigma^+)\bigr).
\]
Combining these inclusions and applying Proposition
\ref{prop:asai-L-space-decomposition}, we obtain
\[
\begin{aligned}
\mathcal{I}_\As\bigl(\sigma\bigr)
&\supseteq
\mathcal{L}\bigl(\As(\sigma^-)\bigr)
+
R_{\chi_1,\sigma^+}(s)\,
\mathcal{L}\bigl(\As(\sigma^+)\bigr)
=
\mathcal{L}\bigl(\As(\sigma)\bigr).
\end{aligned}
\]
This completes the proof.
\end{proof}
    
\end{section}

\begin{section}{Proof of Theorem 1.2}\label{s5}

Let $\sigma$ be an $m$-dimensional admissible representation of $W_F$. Suppose that $\sigma$ decomposes into $\ell \geq 2$ irreducible constituents, which we order by their exponents $e_i$:
\[
  \sigma = \delta_1 \oplus \delta_2 \oplus \cdots \oplus \delta_\ell, \qquad e_1 \leq e_2 \leq \cdots \leq e_\ell.
\]
Define the truncated representations $\sigma^+ = \delta_2\oplus\cdots\oplus\delta_\ell$ and $\sigma^- = \delta_1\oplus\cdots\oplus\delta_{\ell-1}$.

Applying the direct-sum formula for the exterior square by isolating $\delta_1$ gives $\wedge^2\sigma = \Sigma_1 \oplus \Sigma_2$, where
\[
  \Sigma_1 = \wedge^2\delta_1 \oplus (\delta_1\otimes\sigma^+) \qquad\text{and}\qquad \Sigma_2 = \wedge^2\sigma^+.
\]

Alternatively, applying the direct-sum formula by isolating $\delta_\ell$ yields $\wedge^2\sigma = \Theta_1 \oplus \Theta_2$, where
\[
  \Theta_1 = \wedge^2\sigma^- \qquad\text{and}\qquad \Theta_2 = \wedge^2\delta_\ell \oplus (\sigma^-\otimes\delta_\ell).
\]

Thus, we obtain the two decompositions:
\[
  \wedge^2\sigma = \Sigma_1 \oplus \Sigma_2 = \Theta_1 \oplus \Theta_2.
\]

For admissible representations $\delta$ and $\kappa$ of $W_F$, define
\[
  Q_{\delta,\kappa}(s) = \frac{ L\bigl(s,\wedge^2\delta\bigr) L\bigl(s,\delta\otimes\kappa\bigr) }{ L\bigl(1-s,\wedge^2(\delta^\vee)\bigr) L\bigl(1-s,\delta^\vee\otimes\kappa^\vee\bigr) }.
\]

\begin{lemma}
\label{lem:exterior-square-pole-separation}
With the notation above,
\[
\mathcal{P}\bigl(L(1-s,\Sigma_1^\vee)\bigr)
\cap
\mathcal{P}\bigl(L(s,\Theta_2)\bigr)
=
\varnothing.
\]
\end{lemma}

\begin{proof}
The irreducible constituents of \(\Sigma_1\) have exponents among
the numbers \(e_1+e_j\), \(1\leq j\leq\ell\), whereas those of
\(\Theta_2\) have exponents among the numbers \(e_i+e_\ell\),
\(1\leq i\leq\ell\); the argument of
Lemma~\ref{lem:asai-pole-separation} therefore applies verbatim.
\end{proof}

Using Lemma~\ref{lem:exterior-square-pole-separation} and
Proposition~\ref{prop:two-decomposition}, we obtain the following.

\begin{prop}[Exterior-square $L$-space decomposition]
\label{prop:exterior-square-L-space-decomposition}
With the notation above, one has
\[
\mathcal{L}\bigl(\wedge^2\sigma\bigr)
=
\mathcal{L}\bigl(\wedge^2\sigma^-\bigr)
+
Q_{\delta_1,\sigma^+}(s)\,
\mathcal{L}\bigl(\wedge^2\sigma^+\bigr).
\]
Equivalently,
\[
\mathcal{L}\bigl(\wedge^2\sigma\bigr)
=
\mathcal{L}\bigl(\wedge^2\sigma^-\bigr)
+
\frac{
L\bigl(s,\wedge^2\delta_1\bigr)
L\bigl(s,\delta_1\otimes\sigma^+\bigr)
}{
L\bigl(1-s,\wedge^2(\delta_1^\vee)\bigr)
L\bigl(1-s,\delta_1^\vee\otimes(\sigma^+)^\vee\bigr)
}
\mathcal{L}\bigl(\wedge^2\sigma^+\bigr).
\]
\end{prop}

Next, we introduce the partial Jacquet--Shalika integrals. For $r\geq1$, put
\[
\begin{aligned}
h_{2r}(g,X)
&=
\sigma_{2r}
\begin{pmatrix}
I_r&X\\
0&I_r
\end{pmatrix}
\begin{pmatrix}
g&0\\
0&g
\end{pmatrix},\\[2mm]
h_{2r+1}(g,X)
&=
\sigma_{2r+1}
\begin{pmatrix}
I_r&X&0\\
0&I_r&0\\
0&0&1
\end{pmatrix}
\begin{pmatrix}
g&0&0\\
0&g&0\\
0&0&1
\end{pmatrix}.
\end{aligned}
\]
Let $\sigma$ be an admissible representation of $W_F$ with $m_\sigma = \dim(\sigma)$, and fix $W \in \mathcal{W}(\pi_\sigma, \psi_F)$ and $\phi \in \mathcal{S}(F^r)$. For $2r \le m_\sigma$, define
\[
\begin{aligned}
J_{2r}(s,W,\phi) = \int_{N_r(F)\backslash G_r(F)} &\int_{V_r(F)\backslash M_r(F)} W \begin{pmatrix} h_{2r}(g,X) & 0 \\ 0 & I_{m_\sigma-2r} \end{pmatrix} \\
&\quad \times \psi_F(-\operatorname{tr}X) \phi(e_r g) |\det g|_F^{s - (m_\sigma - 2r)} \, dX \, dg.
\end{aligned}
\]
Similarly, if $2r + 1 \le m_\sigma$, define
\[
\begin{aligned}
J_{2r+1}(s,W) = \int_{N_r(F)\backslash G_r(F)} &\int_{V_r(F)\backslash M_r(F)} W \begin{pmatrix} h_{2r+1}(g,X) & 0 \\ 0 & I_{m_\sigma-2r-1} \end{pmatrix} \\
&\quad \times \psi_F(-\operatorname{tr}X) |\det g|_F^{s - (m_\sigma - 2r - 1) - 1} \, dX \, dg.
\end{aligned}
\]
When $2r = m_\sigma$ or $2r + 1 = m_\sigma$, these are precisely the Jacquet--Shalika integrals defining $\mathcal{J}_{\wedge^2}(\sigma)$.

\begin{lemma}
\label{lem:partial-jacquet-shalika-integrals}
Every partial Jacquet--Shalika integral defined above belongs to $\mathcal{J}_{\wedge^2}(\sigma)$.
\end{lemma}

\begin{proof}
We show that every partial integral of rank $q<m_\sigma$ is a finite sum of partial integrals of rank $q+1$. Suppose first that $q=2r$ is even. Choose $\phi^\sharp\in\mathcal{S}(F^{2r})$ such that
\[
  \phi^\sharp(0,y)=\phi(y), \qquad y\in F^r.
\]
By \cite[Proposition~6.1]{Jac2009}, there exists $W'\in\mathcal{W}(\pi_\sigma,\psi_F)$ such that
\[
  W' \begin{pmatrix} g_0&0\\ 0&I_{m_\sigma-2r} \end{pmatrix} = W \begin{pmatrix} g_0&0\\ 0&I_{m_\sigma-2r} \end{pmatrix} \phi^\sharp(e_{2r}g_0)
\]
for every $g_0\in G_{2r}(F)$. Since $e_{2r}h_{2r}(g,X)=(0,e_rg)$ and
\[
  h_{2r+1}(g,X) = \begin{pmatrix} h_{2r}(g,X)&0\\ 0&1 \end{pmatrix},
\]
we obtain
\[
  J_{2r}(s,W,\phi) = J_{2r+1}(s,W').
\]

Suppose next that $q=2r+1$ is odd, and set $d=m_\sigma-2r-1$. Thus $d\geq1$. For $y\in F^r$, put
\[
  u_r(y) = \begin{pmatrix} I_r&0&0\\ 0&I_r&0\\ 0&y&1 \end{pmatrix} \qquad\text{and}\qquad \overline{u}_r(y) = \begin{pmatrix} u_r(y)&0\\ 0&I_{d} \end{pmatrix}.
\]
Applying the Dixmier--Malliavin lemma \cite[Lemma~6.1]{Jac2009} to the subgroup $\{\overline{u}_r(y) : y \in F^r\}$, we may write
\[
  W = \sum_{i=1}^a \int_{F^r} R\bigl(\overline{u}_r(y)\bigr) W_i \, \phi_i(y) \, dy,
\]
where $W_i \in \mathcal{W}(\pi_\sigma, \psi_F)$, and $\phi_i \in C_c^\infty(F^r)$. Consequently, we have
\[
  \begin{aligned}
    J_{2r+1}(s,W) ={}& \sum_{i=1}^a \int_{N_r(F)\backslash G_r(F)} \int_{V_r(F)\backslash M_r(F)} \int_{F^r} W_i \begin{pmatrix} h_{2r+1}(g,X) u_r(y) & 0 \\ 0 & I_{d} \end{pmatrix} \\
    &\quad \times \phi_i(y) \psi_F(-\operatorname{tr} X) |\det g|_F^{s-d-1} \, dy \, dX \, dg.
  \end{aligned}
\]

Choose $\phi_i^\sharp\in\mathcal{S}(F^{2r+1})$ such that $\phi_i^\sharp(0,y,1)=\phi_i(y)$ for $y\in F^r$. Since 
$$e_{2r+1}h_{2r+1}(g,X)u_r(y)=(0,y,1),$$ an application of \cite[Proposition~6.1]{Jac2009} gives $W_i'\in\mathcal{W}(\pi_\sigma,\psi_F)$ such that 
\[
  W_i'\begin{pmatrix} h_{2r+1}(g,X)u_r(y)&0\\ 0&I_{d} \end{pmatrix}= \phi_i(y)W_i\begin{pmatrix} h_{2r+1}(g,X)u_r(y)&0\\ 0&I_{d} \end{pmatrix}.
\]

Set
\[
  \omega_{2r+2} = \begin{pmatrix} \sigma_{2r}^{-1}&0\\ 0&I_2 \end{pmatrix} \sigma_{2r+2}, \qquad \overline{\omega}_{2r+2} = \begin{pmatrix} \omega_{2r+2}&0\\ 0&I_{d-1} \end{pmatrix},
\]
and put $W_i''=R\bigl(\overline{\omega}_{2r+2}^{-1}\bigr)W_i'$. For $Y = \left(\begin{smallmatrix} X&0\\ yg^{-1}&0 \end{smallmatrix}\right)$, a direct matrix calculation gives
\[
  \begin{pmatrix} h_{2r+1}(g,X)u_r(y)&0\\ 0&1 \end{pmatrix} \omega_{2r+2} = h_{2r+2} \left( \begin{pmatrix} g&0\\ 0&1 \end{pmatrix},Y \right).
\]
Moreover, $\operatorname{tr}Y=\operatorname{tr}X$ and $dy=|\det g|_F\,d(yg^{-1})$. It follows that the $i$-th summand is
\[
  \begin{aligned}
    K_{2r+2}(s,W_i'') ={}& \int_{N_r(F)\backslash G_r(F)} \int_{V_{r+1}(F)\backslash M_{r+1}(F)} W_i''\begin{pmatrix} h_{2r+2}\left( \begin{pmatrix}g&0\\0&1\end{pmatrix},Y \right)&0\\ 0&I_{d-1} \end{pmatrix}\\
    &\qquad\qquad\times \psi_F(-\operatorname{tr}Y) |\det g|_F^{s-d}\,dY\,dg.
  \end{aligned}
\]

For $t\in F^\times$, set $a_r(t)=\operatorname{diag}(t,I_{r-1})$, and for $x\in F^r$, put
\[
  b_r(x,t) = \begin{pmatrix} a_r(t)^{-1}&0&0&0\\ x&t&0&0\\ 0&0&a_r(t)^{-1}&0\\ 0&0&x&t \end{pmatrix} \in G_{2r+2}(F).
\]
Applying the Dixmier--Malliavin lemma to the subgroup consisting of
\[
  \overline{b}_r(x,t) = \begin{pmatrix} b_r(x,t)&0\\ 0&I_{d-1} \end{pmatrix}, \qquad x\in F^r, \quad t\in F^\times,
\]
we may write
\[
  W_i'' = \sum_{j=1}^{a_i} \int_{F^r\times F^\times} R\bigl(\overline{b}_r(x,t)\bigr)W_{ij}\, \phi_{ij}(x,t)\,dx\,|t|_F\,d^\times t,
\]
where $W_{ij}\in\mathcal{W}(\pi_\sigma,\psi_F)$ and $\phi_{ij}\in C_c^\infty(F^r\times F^\times)$. Substituting this expression into $K_{2r+2}(s,W_i'')$ and making the change of variables $ga_r(t)^{-1}\mapsto g$, we obtain
\[
  \begin{aligned}
    K_{2r+2}(s,W_i'') ={}& \sum_{j=1}^{a_i} \int_{F^\times} \int_{F^r} \int_{N_r(F)\backslash G_r(F)} \int_{V_{r+1}(F)\backslash M_{r+1}(F)} W_{ij}\begin{pmatrix} h_{2r+2}\left( \begin{pmatrix}g&0\\x&t\end{pmatrix},Y \right)&0\\ 0&I_{d-1} \end{pmatrix}\\
    &\quad\times \phi_{ij}(x,t)\psi_F(-\operatorname{tr}Y) |\det g|_F^{s-d}|t|_F^{s-d+1} \,dY\,dg\,dx\,d^\times t.
  \end{aligned}
\]

Define $\phi_{ij}^0\in\mathcal{S}(F^{r+1})$ by
\[
  \phi_{ij}^0(x,t) = \begin{cases} \phi_{ij}(x,t),&t\neq0,\\ 0,&t=0. \end{cases}
\]
This is a Schwartz function because $\phi_{ij}$ has compact support in $F^r\times F^\times$. With compatible Haar measures, we have the identity
\[
  \int_{N_{r+1}(F)\backslash G_{r+1}(F)}f(\gamma)\,d\gamma = \int_{F^\times} \int_{F^r} \int_{N_r(F)\backslash G_r(F)} f\begin{pmatrix}g&0\\x&t\end{pmatrix} |\det g|_F^{-1}\,dg\,dx\,d^\times t.
\]
Since
\[
  \left| \det\begin{pmatrix}g&0\\x&t\end{pmatrix} \right|_F^{s-(d-1)} |\det g|_F^{-1} = |\det g|_F^{s-d}|t|_F^{s-d+1},
\]
it follows that
\[
  K_{2r+2}(s,W_i'') = \sum_{j=1}^{a_i}J_{2r+2}(s,W_{ij},\phi_{ij}^0).
\]
Thus every partial integral of rank $q<m_\sigma$ is a finite sum of partial integrals of rank $q+1$. Iterating the two steps completes the proof.
\end{proof}

We now turn to the inductive reduction of Jacquet--Shalika integrals. Let $\kappa$ and $\delta$ be admissible representations of $W_F$, and assume $m_\kappa = \dim(\kappa) \ge 2$ and $m_\delta = \dim(\delta)$. 

\begin{prop}[Jacquet--Shalika integral reduction]
\label{prop:exterior-square-integral-inclusions}
The following inclusions hold:
\begin{enumerate}[label=(\arabic*)]
    \item If $\kappa \preceq \delta$, then $\mathcal{J}_{\wedge^2}(\kappa \oplus \delta) \supseteq \mathcal{J}_{\wedge^2}(\kappa)$.
    \item If $\delta \preceq \kappa$, then $\mathcal{J}_{\wedge^2}(\delta \oplus \kappa) \supseteq Q_{\delta,\kappa}(s)\mathcal{J}_{\wedge^2}(\kappa)$.
\end{enumerate}
\end{prop}
\begin{proof}
Assume first that $\kappa\preceq\delta$. Suppose that $m_\kappa=2r$. Let
\[
J(s,W_\kappa,\phi)
\in
\mathcal{J}_{\wedge^2}(\kappa),
\]
where $W_\kappa\in\mathcal{W}(\pi_\kappa,\psi_F),\>
\phi\in\mathcal{S}(F^r).$
Choose $\phi_0\in\mathcal{S}(F^{2r})$ such that
\[
\phi_0(0,y)=\phi(y),
\qquad y\in F^r.
\]
By \cite[Proposition~14.1]{Jac2009}, there exists
$W^0\in\mathcal{W}(\pi_{\kappa\oplus\delta},\psi_F)$ such that
\[
W^0
\begin{pmatrix}
g_0&0\\
0&I_{m_\delta}
\end{pmatrix}
=
W_\kappa(g_0)\phi_0(e_{2r}g_0)
|\det g_0|_F^{m_\delta/2}
\]
for every $g_0\in G_{2r}(F)$. For $g_0=h_{2r}(g,X)$, one has
\[
e_{2r}g_0=(0,e_rg)
\qquad\text{and}\qquad
|\det g_0|_F^{m_\delta/2}=|\det g|_F^{m_\delta}.
\]
Therefore,
\[
\begin{aligned}
J(s,W_\kappa,\phi)
={}&
\int_{N_r(F)\backslash G_r(F)}
\int_{V_r(F)\backslash M_r(F)}
W^0\left(
\begin{array}{cc}
h_{2r}(g,X)&0\\
0&I_{m_\delta}
\end{array}
\right)\\
&\qquad\qquad\times
\psi_F(-\operatorname{tr}X)
|\det g|_F^{s-m_\delta}\,dX\,dg.
\end{aligned}
\]
By \cite[Proposition~6.1]{Jac2009}, there
exist finitely many $W_i\in\mathcal{W}(\pi_{\kappa\oplus\delta},\psi_F), \phi_i\in\mathcal{S}(F^{2r})\>
$
such that
\[
W^0
\begin{pmatrix}
g_0&0\\
0&I_{m_\delta}
\end{pmatrix}
=
\sum_i
W_i
\begin{pmatrix}
g_0&0\\
0&I_{m_\delta}
\end{pmatrix}
\phi_i(e_{2r}g_0).
\]
Set
\[
\phi_i^0(y)=\phi_i(0,y),
\qquad y\in F^r.
\]
Then
\[
J(s,W_\kappa,\phi)
=
\sum_iJ_{2r}(s,W_i,\phi_i^0).
\]
Lemma~\ref{lem:partial-jacquet-shalika-integrals} now gives
\[
J(s,W_\kappa,\phi)
\in
\mathcal{J}_{\wedge^2}(\kappa\oplus\delta).
\]

Suppose next that $m_\kappa=2r+1$. Let
\[
J(s,W_\kappa)
\in
\mathcal{J}_{\wedge^2}(\kappa).
\]
Choose $\phi_0\in\mathcal{S}(F^{2r+1})$ such that $\phi_0(e_{2r+1})=1.$ By \cite[Proposition~14.1]{Jac2009}, there exists
$W\in\mathcal{W}(\pi_{\kappa\oplus\delta},\psi_F)$ such that
\[
W
\begin{pmatrix}
g_0&0\\
0&I_{m_\delta}
\end{pmatrix}
=
W_\kappa(g_0)\phi_0(e_{2r+1}g_0)
|\det g_0|_F^{m_\delta/2}.
\]
For $g_0=h_{2r+1}(g,X)$, one has
\[
e_{2r+1}g_0=e_{2r+1}
\qquad\text{and}\qquad
|\det g_0|_F^{m_\delta/2}=|\det g|_F^{m_\delta}.
\]
Consequently,
\[
J(s,W_\kappa)=J_{2r+1}(s,W),
\]
and the first inclusion follows again from
Lemma~\ref{lem:partial-jacquet-shalika-integrals}.

For the second inclusion, assume that $\delta\preceq\kappa$. Then $\kappa^\vee\preceq\delta^\vee.$ Applying the first inclusion to
\[
(\delta\oplus\kappa)^\vee
=
\kappa^\vee\oplus\delta^\vee
\]
gives
\[
\mathcal{J}_{\wedge^2}
\bigl((\delta\oplus\kappa)^\vee\bigr)
\supseteq
\mathcal{J}_{\wedge^2}(\kappa^\vee).
\]
Moreover,
\[
\wedge^2(\delta\oplus\kappa)
\simeq
\wedge^2\delta
\oplus
(\delta\otimes\kappa)
\oplus
\wedge^2\kappa.
\]
Using \eqref{shalikafe2},  
$\mathcal{J}_{\wedge^2}(\delta\oplus\kappa)$ contains
\[
\gamma\bigl(s,\wedge^2\delta,\psi_F\bigr)^{-1}
\gamma\bigl(s,\delta\otimes\kappa,\psi_F\bigr)^{-1}
\mathcal{J}_{\wedge^2}(\kappa).
\]
Using
\[
\gamma(s,\rho,\psi_F)
=
\varepsilon(s,\rho,\psi_F)
\frac{L(1-s,\rho^\vee)}{L(s,\rho)},
\]
and the fact that $\mathcal{J}_{\wedge^2}(\kappa)$ is invariant under multiplication by functions of the form \(cA^s\), where
\(c\in\mathbb{C}^{\times}\) and \(A>0\), we obtain
\[
\mathcal{J}_{\wedge^2}(\delta\oplus\kappa)
\supseteq
\frac{
L\bigl(s,\wedge^2\delta\bigr)
L\bigl(s,\delta\otimes\kappa\bigr)
}{
L\bigl(1-s,\wedge^2(\delta^\vee)\bigr)
L\bigl(1-s,\delta^\vee\otimes\kappa^\vee\bigr)
}
\mathcal{J}_{\wedge^2}(\kappa).\qedhere
\]
\end{proof}

To prove Theorem~\ref{main2}, it suffices to show that
\[
\mathcal{J}_{\wedge^2}(\sigma)
\supseteq
\mathcal{L}\bigl(\wedge^2\sigma\bigr).
\]
We first prove this when \(\dim(\sigma)=2\) or \(3\).

\begin{prop}
\label{prop:exterior-square-rank-two}
Let $\sigma$ be a two-dimensional admissible representation of $W_F$.
Then
\[
\mathcal{J}_{\wedge^2}(\sigma)
\supseteq
\mathcal{L}\bigl(\wedge^2\sigma\bigr).
\]
\end{prop}

\begin{proof}
Let $\pi=\pi_\sigma$. Since $\wedge^2\sigma=\det\sigma$, the character corresponding to $\wedge^2\sigma$ under local class field theory is $\omega_\pi$. Write $\omega_\pi = \chi |\cdot|_F^{t}$ where $\chi$ is a normalized character and $t\in\mathbb{C}$. Then
\[
L(s,\wedge^2\sigma) = L(s+t,\chi).
\]

Let $F \in \mathcal{L}(\wedge^2\sigma)$, and set $F_0(s) = F(s-t)$. Translation in $s$ preserves the defining vertical-strip estimates, so $F_0 \in \mathcal{L}(\chi)$. Because $\chi$ is normalized, \cite[Lemma~12.3]{Jac2009} guarantees the existence of $\phi \in \mathcal{S}(F)$ such that
\[
F_0(s)
=
\int_{F^\times}
\phi(a)\chi(a)|a|_F^s\,d^\times a.
\]
Replacing $s$ by $s+t$ gives
\[
F(s)
=
\int_{F^\times}
\phi(a)\chi(a)|a|_F^{s+t}\,d^\times a
=
\int_{F^\times}
\phi(a)\omega_\pi(a)|a|_F^s\,d^\times a.
\]
Moreover, $\sigma_2=I_2$ and $V_1(F)=M_1(F)$. Choose $W \in \mathcal{W}(\pi,\psi_F)$ such that $W(I_2)=1$. The right-hand side is then exactly $J(s,W,\phi)$, which proves $F \in \mathcal{J}_{\wedge^2}(\sigma)$.
\end{proof}

\begin{prop}
\label{prop:exterior-square-rank-three}
Let \(\sigma\) be a three-dimensional admissible representation of
\(W_F\). Then
\[
\mathcal{J}_{\wedge^2}(\sigma)
\supseteq
\mathcal{L}\bigl(\wedge^2\sigma\bigr).
\]
\end{prop}

\begin{proof}
Let \(\pi=\pi_\sigma\). Since
\[
\wedge^2\sigma
\simeq
\sigma^\vee\otimes\det\sigma
\]
and \(\omega_\pi\) corresponds to \(\det\sigma\), the generic induced
representation attached to \(\wedge^2\sigma\) is
\[
\pi_{\wedge^2\sigma}
\simeq
\pi_{\sigma^\vee}\otimes\omega_\pi.
\]

For \(W\in\mathcal{W}(\pi_\sigma,\psi_F)\), define
\[
W^\sharp(g)
=
\omega_\pi(\det g)\,
\bigl(R(w_3)\widetilde W\bigr)(g),
\qquad g\in G_3(F).
\]
The map
\[
W\longmapsto\widetilde W
\]
is a bijection from
\(\mathcal{W}(\pi_\sigma,\psi_F)\) onto
\(\mathcal{W}(\pi_{\sigma^\vee},\psi_F^{-1})\).
Right translation by \(w_3\), followed by twisting by
\(\omega_\pi\circ\det\), therefore gives a bijection
\[
\mathcal{W}(\pi_\sigma,\psi_F)
\longrightarrow
\mathcal{W}(\pi_{\wedge^2\sigma},\psi_F^{-1}),
\qquad
W\longmapsto W^\sharp.
\]

Since \(\sigma_3=I_3\) and \(V_1(F)=M_1(F)\), we have
\[
J(s,W)
=
\int_{F^\times}
W\begin{pmatrix}
a&0&0\\
0&a&0\\
0&0&1
\end{pmatrix}
|a|_F^{s-1}\,d^\times a.
\]
Moreover,
\[
\begin{aligned}
W^\sharp
\begin{pmatrix}
a&0&0\\
0&1&0\\
0&0&1
\end{pmatrix}
&=
\omega_\pi(a)\,
W\begin{pmatrix}
1&0&0\\
0&1&0\\
0&0&a^{-1}
\end{pmatrix}\\
&=
W\begin{pmatrix}
a&0&0\\
0&a&0\\
0&0&1
\end{pmatrix}.
\end{aligned}
\]
Consequently,
\[
J(s,W)
=
\int_{F^\times}
W^\sharp
\begin{pmatrix}
a&0&0\\
0&1&0\\
0&0&1
\end{pmatrix}
|a|_F^{s-1}\,d^\times a.
\]

Let \(f\in\mathcal{L}(\wedge^2\sigma)\). Applying
\cite[Theorem~2.6(i)]{Jac2009}, with additive character
\(\psi_F^{-1}\), to the pair
$(\wedge^2\sigma,\mathbf{1})$,
we obtain
\[
W^\sharp
\in
\mathcal{W}(\pi_{\wedge^2\sigma},\psi_F^{-1})
\]
such that
\[
f(s)
=
\int_{F^\times}
W^\sharp
\begin{pmatrix}
a&0&0\\
0&1&0\\
0&0&1
\end{pmatrix}
|a|_F^{s-1}\,d^\times a.
\]
Choose \(W\in\mathcal{W}(\pi_\sigma,\psi_F)\) mapping to
\(W^\sharp\). Then
\[
f(s)=J(s,W),
\]
and therefore
$f\in\mathcal{J}_{\wedge^2}(\sigma).$
\end{proof}
We now prove the general case by induction.

\begin{proof}[Proof of Theorem~\ref{main2}]
We argue by induction on $m=\dim(\sigma)$. The cases $m=2$ and $m=3$ follow from Propositions~\ref{prop:exterior-square-rank-two} and~\ref{prop:exterior-square-rank-three}, respectively. Assume that $m>3$, and that the result holds for every admissible representation of $W_F$ of dimension $d$ with $2 \leq d < m$.

Using the notation introduced at the beginning of this section, we write $\sigma = \sigma^-\oplus\delta_\ell = \delta_1\oplus\sigma^+$. Since $F$ is archimedean, every irreducible admissible representation of $W_F$ has dimension at most two \cite[p.~403]{Kna1994}. Consequently, $\dim(\sigma^-) \geq m-2 \geq 2$ and $\dim(\sigma^+) \geq m-2 \geq 2$, so the induction hypothesis applies to both subrepresentations.

Since $\sigma^-\preceq\delta_\ell$, Proposition~\ref{prop:exterior-square-integral-inclusions} yields 
$$\mathcal{J}_{\wedge^2}(\sigma) \supseteq \mathcal{J}_{\wedge^2}(\sigma^-).$$
Applying the induction hypothesis to $\sigma^-$ gives
\[
\mathcal{J}_{\wedge^2}(\sigma)
\supseteq
\mathcal{L}\bigl(\wedge^2\sigma^-\bigr).
\]
Similarly, since $\delta_1\preceq\sigma^+$, Proposition~\ref{prop:exterior-square-integral-inclusions} and the induction hypothesis imply
\[
\mathcal{J}_{\wedge^2}(\sigma)
\supseteq
Q_{\delta_1,\sigma^+}(s)\,
\mathcal{J}_{\wedge^2}(\sigma^+)
\supseteq
Q_{\delta_1,\sigma^+}(s)\,
\mathcal{L}\bigl(\wedge^2\sigma^+\bigr).
\]

Combining these two inclusions and applying Proposition~\ref{prop:exterior-square-L-space-decomposition}, we conclude that
\[
\mathcal{J}_{\wedge^2}(\sigma)
\supseteq
\mathcal{L}\bigl(\wedge^2\sigma^-\bigr)
+
Q_{\delta_1,\sigma^+}(s)\,
\mathcal{L}\bigl(\wedge^2\sigma^+\bigr)
=
\mathcal{L}\bigl(\wedge^2\sigma\bigr).
\]
This completes the proof.
\end{proof}
\end{section}

\begin{section}{Local Applications}\label{s6}

In this section, we give some local applications of
Theorems~\ref{main1} and~\ref{main2}. We study the normalized Flicker
and Jacquet--Shalika functionals associated with the local zeta
integrals. In particular, we characterize
\(G_n(\mathbb{R})\)-distinction by the occurrence of an exceptional
pole.

\begin{thm}\label{continuity}
Let
$
\pi\in\operatorname{Irr}_{\mathrm{gen}}
\bigl(G_n(\mathbb{C})\bigr).
$
For \(s\in\mathbb{C}\) and
\(\phi\in\mathcal{S}(\mathbb{R}^n)\), let
\(\Lambda_{s,\phi}\) be the linear functional on
\(\mathcal{W}(\pi,\psi_{\mathbb{C}})\) defined by
\[
\Lambda_{s,\phi}(W)
=
\frac{I(s,W,\phi)}{L(s,\pi,\As)}.
\]
Then, for every fixed \(\phi\in\mathcal{S}(\mathbb{R}^n)\), the
functional \(\Lambda_{s,\phi}\) is continuous on
\(\mathcal{W}(\pi,\psi_{\mathbb{C}})\), uniformly for \(s\) in compact
subsets of \(\mathbb{C}\). Moreover, for every \(s_0\in\mathbb{C}\), there exist
\(W_0\in\mathcal{W}(\pi,\psi_{\mathbb{C}})\) and
\(\phi_0\in\mathcal{S}(\mathbb{R}^n)\) such that
\[
\Lambda_{s_0,\phi_0}(W_0)\neq 0.
\]
\end{thm}

\begin{proof}
The proof follows the argument of \cite[Theorem~1.1]{CPS2004}; we include the details for completeness. Fix \(\phi\in\mathcal{S}(\mathbb{R}^n)\) and put
\(\mathcal{W}_{\pi}=\mathcal{W}(\pi,\psi_{\mathbb{C}})\).
By \cite[Theorem~1(iii)]{BP2021}, the function
\(s\mapsto\Lambda_{s,\phi}(W)\) is entire for every
\(W\in\mathcal{W}_{\pi}\).

We first prove continuity in the half-plane of absolute convergence.
Let \(\Omega\) be a compact subset of this half-plane and suppose that
\(W_k\to0\) in \(\mathcal{W}_{\pi}\). By the extended
Dixmier--Malliavin lemma \cite[Proposition~1.1]{CPS2004}, there exist
\(f_1,\ldots,f_r\in C_c^\infty(G_n(\mathbb{C}))\) and
\(W_{k,j}\to0\) in \(\mathcal{W}_{\pi}\) such that
\[
W_k=\sum_{j=1}^rR(f_j)W_{k,j}.
\]
For each \(j\), the standard gauge estimate gives a continuous
seminorm \(\beta_j\) on \(\mathcal{W}_{\pi}\) and a gauge \(\xi_j\)
such that
\[
|R(f_j)W(g)|\leq\beta_j(W)\xi_j(g),
\qquad g\in G_n(\mathbb{R}).
\]
It follows that
\[
\sup_{s\in\Omega}|I(s,W_k,\phi)|
\leq
\sum_{j=1}^r\beta_j(W_{k,j})
\sup_{s\in\Omega}
\int_{N_n(\mathbb{R})\backslash G_n(\mathbb{R})}
\xi_j(g)|\phi(e_ng)|
|\det g|_{\mathbb{R}}^{\Re(s)}\,dg.
\]
The integrals on the right are finite, uniformly for
\(s\in\Omega\), while \(\beta_j(W_{k,j})\to0\). Therefore
\(I(s,W_k,\phi)\to0\) uniformly on \(\Omega\). Since
\(L(s,\pi,\As)^{-1}\) is bounded on \(\Omega\), the same is true of
\(\Lambda_{s,\phi}(W_k)\). Thus \(\Lambda_{s,\phi}\) is continuous in
the half-plane of absolute convergence, uniformly for \(s\) in
compact subsets.

Set
\[
\Lambda'_{s,\phi}(W)=e^{s^2}\Lambda_{s,\phi}(W).
\]
For a real number \(\sigma\), write
\(I(\sigma;|W|,|\phi|)\) for the corresponding absolute-value
integral at \(s=\sigma\). Choose \(B\) in the half-plane of absolute
convergence. On the line \(\Re(s)=B\), Stirling's formula gives a
constant \(c_B>0\) such that
\[
|\Lambda'_{s,\phi}(W)|
\leq c_B I(B;|W|,|\phi|).
\tag{1}
\]

Using Theorem \ref{fe1} and applying the preceding argument to \(\pi^{\vee}\) and
\(\widehat\phi\), we obtain continuity in a left half-plane. Choose
\(A\) in this half-plane. Stirling's formula and the explicit form of
the gamma factor give a constant \(c_A>0\) such that, on
\(\Re(s)=A\),
\[
|\Lambda'_{s,\phi}(W)|
\leq
c_A I(1-A;|\widetilde W|,|\widehat\phi|).
\tag{2}
\]

Consider now the strip \(A\leq\Re(s)\leq B\). As \(\Lambda_{s,\phi}(W)\) is of finite order in vertical strips \cite[Theorem~1(iii)]{BP2021}, \(\Lambda'_{s,\phi}(W)\) satisfies the growth condition required for
the Phragmén--Lindelöf principle. From \((1)\) and \((2)\), we obtain
\[
|\Lambda'_{s,\phi}(W)|
\leq
\max\left\{
c_B I(B;|W|,|\phi|),
c_A I(1-A;|\widetilde W|,|\widehat\phi|)
\right\}
\]
throughout the strip.

Now let \(W_k\to0\) in \(\mathcal{W}_{\pi}\). Since
\(W\mapsto\widetilde W\) is continuous, \(\widetilde W_k\to0\).
The same argument as in the half-plane of absolute convergence gives
\[
I(B;|W_k|,|\phi|)\longrightarrow0,
\qquad
I(1-A;|\widetilde W_k|,|\widehat\phi|)\longrightarrow0.
\]
Consequently, \(\Lambda'_{s,\phi}(W_k)\to0\) uniformly in the strip.
Since \(e^{-s^2}\) is bounded on compact subsets, it follows that
\(\Lambda_{s,\phi}(W_k)\to0\) uniformly for \(s\) in compact subsets
of \(\mathbb{C}\). This proves the continuity assertion.

Finally, let \(s_0\in\mathbb{C}\). By Theorem~\ref{main1}, there exist
\(W_j\in\mathcal{W}(\pi,\psi_{\mathbb{C}})\) and
\(\phi_j\in\mathcal{S}(\mathbb{R}^n)\), \(1\leq j\leq r\), such that
\[
L(s,\pi,\As)=\sum_{j=1}^r I(s,W_j,\phi_j).
\]
Dividing by \(L(s,\pi,\As)\) and using holomorphic continuation gives
\(1=\sum_{j=1}^r\Lambda_{s,\phi_j}(W_j)\). Hence, for some \(j\),
\(\Lambda_{s_0,\phi_j}(W_j)\neq0\).
\end{proof}

\begin{remark}\label{poleorder}
By Theorem~\ref{continuity}, if \(s_0\) is a pole of
\(L(s,\pi,\As)\), then its order is the maximal pole order for the
family of Flicker integrals $\mathcal{I}_\As(\sigma_\pi)$ at \(s_0\).
\end{remark}

\begin{lemma}\label{mirabolic-integral}
The map
\[
W \longmapsto
\frac{\mathcal{Z}(s,W)}{L(s,\pi,\As)}
=
\frac{1}{L(s,\pi,\As)}
\int_{N_n(\mathbb{R})\backslash P_n(\mathbb{R})}
W(p)\,|\det p|_{\mathbb{R}}^{s-1}\,dp,
\]
is a well-defined continuous linear functional on
\(\mathcal{W}(\pi,\psi_{\mathbb{C}})\).
\end{lemma}

\begin{proof}
By Lemma~\ref{lem:partial-flicker-integrals}, the above quotient is 
well-defined. The remaining part of the proof is identical to that of \cite[Proposition~3.2]{Chai2015}, so we omit the details.
\end{proof}
   
Using Theorem~\ref{continuity}, we characterize
\(G_n(\mathbb{R})\)-distinguished representations in terms of exceptional
poles of the Asai \(L\)-function. A complementary characterization in
terms of Langlands parameters is given in \cite[Theorem 1.1]{PWZ2025}. We recall the relevant definitions.

\begin{defn}
A representation \(\pi\) of \(G_n(\mathbb{C})\) is said to be
\(G_n(\mathbb{R})\)-\emph{distinguished} (or simply distinguished) if
\[
\operatorname{Hom}_{G_n(\mathbb{R})}(\pi,\mathbf{1})\neq0.
\]
\end{defn}

We record the following result due to Kemarsky \cite{Kemarsky2015}.

\begin{prop}\label{mirabolic-invariance}
Let $\pi \in \operatorname{Irr}\bigl(G_n(\mathbb{C})\bigr)
$ be a $G_n(\mathbb{R})$-distinguished representation. Then every \(P_n(\mathbb{R})\)-invariant
linear form on the space of \(\pi\) is \(G_n(\mathbb{R})\)-invariant.
\end{prop}

For \(m \geq 0\), define
\[
\mathcal{S}^{m}(\mathbb{R}^{n})
=
\left\{
\phi \in \mathcal{S}(\mathbb{R}^{n})
:
\partial^{\alpha}\phi(0)=0
\text{ for all multi-indices } \alpha \text{ with } |\alpha|<m
\right\}.
\]
Thus,
\[
\mathcal{S}^{0}(\mathbb{R}^{n})=\mathcal{S}(\mathbb{R}^{n})\quad\text{and}\quad 
\mathcal{S}^{1}(\mathbb{R}^{n})
=
\left\{
\phi \in \mathcal{S}(\mathbb{R}^{n})
:
\phi(0)=0
\right\}.
\]

Let \(s_0\) be a pole of \(L(s,\pi,\As)\) of order \(d\). Then we have a
Laurent expansion
\[
I(s,W,\phi)
=
\frac{B_{s_0}(W,\phi)}{(s-s_0)^d}
+\cdots,
\]
where \(B_{s_0}\) is a bilinear form on
$\mathcal{W}(\pi,\psi_{\mathbb{C}})
\times
\mathcal{S}(\mathbb{R}^{n}).$ Observe that
\[ B_{s_0}(W,\phi) = \lim_{s\to s_0} (s-s_0)^d L(s,\pi,\As)\Lambda_{s,\phi}(W), \] and hence  \(B_{s_0}\) is continuous.

\begin{defn}
We say that \(s=s_0\) is an \emph{exceptional pole with level \(m\)} of
\(L(s,\pi,\As)\) if the corresponding bilinear form \(B_{s_0}\) vanishes
on $\mathcal{S}^{m+1}(\mathbb{R}^{n})$
but is not identically zero on $\mathcal{S}^{m}(\mathbb{R}^{n})$.
\end{defn}

\begin{thm}[Distinction and exceptional poles]\label{thm:distinguished-exceptional}
Let $
\pi\in\operatorname{Irr}_{\mathrm{gen}}
\bigl(G_n(\mathbb{C})\bigr).
$
Then \(\pi\) is \(G_n(\mathbb{R})\)-distinguished if and only if
\(s=0\) is an exceptional pole with level $0$ for \(L(s,\pi,\As)\).
\end{thm}

\begin{proof}
Since
\[
\mathcal{S}^{1}(\mathbb{R}^{n})
=
\ker\!\left(
\operatorname{ev}_{0}:
\mathcal{S}(\mathbb{R}^{n})\longrightarrow\mathbb{C}
\right),
\qquad
\operatorname{ev}_{0}(\phi)=\phi(0),
\]
the space \(\mathcal{S}^{1}(\mathbb{R}^{n})\) has codimension one in
\(\mathcal{S}(\mathbb{R}^{n})\). Consequently, if \(s=s_{0}\) is an
exceptional pole with level \(0\) of \(L(s,\pi,\As)\), then there exists a
nonzero continuous linear form
\[
\lambda_{s_{0}}:
\mathcal{W}(\pi,\psi_{\mathbb{C}})
\longrightarrow
\mathbb{C}
\]
such that
\[
B_{s_{0}}(W,\phi)
=
\lambda_{s_{0}}(W)\phi(0).
\]
The equivariance property of \(B_{s_{0}}\) gives
\[
\lambda_{s_{0}}(g\cdot W)
=
|\det g|_{\mathbb{R}}^{-s_{0}}
\lambda_{s_{0}}(W),
\qquad
g\in G_n(\mathbb{R}).
\]
In particular, if \(s_{0}=0\), then \(\lambda_{0}\) is
\(G_n(\mathbb{R})\)-invariant, and hence \(\pi\) is
\(G_n(\mathbb{R})\)-distinguished.

Conversely, suppose that \(\pi\) is \(G_n(\mathbb{R})\)-distinguished.
We prove that \(s=0\) is an exceptional pole with level \(0\) of
\(L(s,\pi,\As)\). When we wish to emphasize the underlying representation \(\pi\), we write
\(\Lambda_{s,\phi}^{\pi}\) instead of \(\Lambda_{s,\phi}\).

For \(\operatorname{Re}(s)\ll 0\), consider
\[
I(1-s,\widetilde{W},\widehat{\phi})
=
\int_{N_n(\mathbb{R})\backslash G_n(\mathbb{R})}
\widetilde{W}(g)\widehat{\phi}(e_ng)
|\det g|_{\mathbb{R}}^{1-s}\,dg,
\]
where
$\widetilde{W}
\in
\mathcal{W}({\pi}^\vee,\psi_{\mathbb{C}}^{-1})
\>\text{and}\>
\phi\in\mathcal{S}(\mathbb{R}^{n}).$ Since \(\pi\) is distinguished, $\omega_\pi$ is trivial on
\(\mathbb{R}^{\times}\).

We first assume that \(\widetilde{W}\) is
\(K_n\)-finite. Then there exist finitely many Whittaker functions
\(\widetilde{W}_i\) and continuous functions \(f_i\) on \(K_n\) such
that
\[
\widetilde{W}(gk)
=
\sum_i f_i(k)\widetilde{W}_i(g),
\qquad
g\in G_n(\mathbb{R}),\quad k\in K_n.
\]
Using the Iwasawa decomposition, we obtain
\[
I(1-s,\widetilde{W},\widehat{\phi})
=
\sum_i
\mathcal{Z}(1-s,\widetilde{W}_i)
\int_{K_n}\int_{\mathbb{R}^{\times}}
f_i(k)\widehat{\phi}(e_nak)
|a|_{\mathbb{R}}^{n(1-s)}
\,d^{\times}a\,dk.
\]
Dividing by the Asai \(L\)-function gives
\[
\Lambda_{1-s,\widehat{\phi}}^{{\pi}^\vee}
(\widetilde{W})
=
\sum_i
\frac{\mathcal{Z}(1-s,\widetilde{W}_i)}
     {L(1-s,{\pi}^\vee,\As)}
\int_{K_n}\int_{\mathbb{R}^{\times}}
f_i(k)\widehat{\phi}(e_nak)
|a|_{\mathbb{R}}^{n(1-s)}
\,d^{\times}a\,dk.
\]
There exists \(\varepsilon>0\) such that the integrals on the
right-hand side are absolutely convergent and holomorphic for
\(\operatorname{Re}(s)<\varepsilon\). Evaluating at \(s=0\), we get
\[
\Lambda_{1,\widehat{\phi}}^{{\pi}^\vee}
(\widetilde{W})
=
\int_{K_n}
\ell\bigl({\pi}^\vee(k)\widetilde{W}\bigr)
\int_{\mathbb{R}^{\times}}
\widehat{\phi}(e_nak)|a|_{\mathbb{R}}^{n}
\,d^{\times}a\,dk,
\]
where
\[
\ell(\widetilde{W})
=
\left.
\frac{\mathcal{Z}(s,\widetilde{W})}
     {L(s,{\pi}^\vee,\As)}
\right|_{s=1}.
\]
The linear form \(\ell\) is \(P_n(\mathbb{R})\)-invariant.
Since \({\pi}^\vee\) is also \(G_n(\mathbb{R})\)-distinguished,
by Proposition \ref{mirabolic-invariance}, \(\ell\) is in fact
\(G_n(\mathbb{R})\)-invariant. Hence
\[
\Lambda_{1,\widehat{\phi}}^{{\pi}^\vee}
(\widetilde{W})
=
\ell(\widetilde{W})
\int_{K_n}\int_{\mathbb{R}^{\times}}
\widehat{\phi}(e_nak)|a|_{\mathbb{R}}^{n}
\,d^{\times}a\,dk.
\]
Fourier inversion with suitably normalized Haar measures gives
\[
\int_{K_n}\int_{\mathbb{R}^{\times}}
\widehat{\phi}(e_nak)|a|_{\mathbb{R}}^{n}
\,d^{\times}a\,dk
=
\phi(0).
\]
Therefore,
\[
\Lambda_{1,\widehat{\phi}}^{{\pi}^\vee}
(\widetilde{W})
=
\ell(\widetilde{W})\phi(0).
\]
Since the \(K_n\)-finite vectors are dense and both sides are
continuous (Theorem \ref{continuity} and Lemma \ref{mirabolic-integral}), this equality holds for every
\(\widetilde{W}\in
\mathcal{W}({\pi}^\vee,\psi_{\mathbb{C}}^{-1})\).

By Theorem \ref{fe1}, there exists
\(\alpha\in\mathbb{C}^{\times}\) such that
\[
\Lambda_{1,\widehat{\phi}}^{{\pi}^\vee}
(\widetilde{W})
=
\alpha\Lambda_{0,\phi}^{\pi}(W).
\]
It follows that
\[
\Lambda_{0,\phi}^{\pi}(W)
=
\alpha^{-1}\ell(\widetilde{W})\phi(0).
\]
In particular, $
\Lambda_{0,\phi}^{\pi}(W)=0
\>
\text{for every }
W\in\mathcal{W}(\pi,\psi_{\mathbb{C}})
\text{ and }
\phi\in\mathcal{S}^{1}(\mathbb{R}^{n}).
$

By \cite[Proposition~3.8]{Yad2024}, one can choose $W_1\in\mathcal{W}(\pi,\psi_{\mathbb{C}})
\>\text{and}\> \phi_1\in\mathcal{S}^{1}(\mathbb{R}^{n})$ such that \(I(s,W_1,\phi_1)\) does not vanish at \(s=0\). However as $\Lambda_{0,\phi_1}^{\pi}(W_1)=0$, \(s=0\) is a pole of
\(L(s,\pi,\As)\) of order say $d$.

Put
\(c_0=\lim_{s\to0}s^dL(s,\pi,\As)\in\mathbb{C}^{\times}\). Then
\[
B_0(W,\phi)
=
c_0\Lambda_{0,\phi}^{\pi}(W)
=
c_0\alpha^{-1}\ell(\widetilde W)\phi(0).
\]
By Remark~\ref{poleorder}, \(d\) is the
maximal pole order among the Flicker integrals at \(s=0\), and hence
\(B_0\) is nonzero. Since \(\phi(0)=0\) for every
\(\phi\in\mathcal{S}^{1}(\mathbb{R}^{n})\), the bilinear form \(B_0\)
vanishes on
\(\mathcal{W}(\pi,\psi_{\mathbb{C}})
\times\mathcal{S}^{1}(\mathbb{R}^{n})\). Thus \(s=0\) is an
exceptional pole with level \(0\) of \(L(s,\pi,\As)\).
\end{proof}

The following continuity result is proved in the same way as
Theorem~\ref{continuity}, with Flicker integrals replaced by
Jacquet--Shalika integrals. We therefore omit the proof.

\begin{thm}\label{continuity-exterior-square}
Let
\[
\pi\in\operatorname{Irr}_{\mathrm{gen}}\bigl(G_m(F)\bigr).
\]
Define the normalized Jacquet--Shalika functionals by
\[
\Sigma_{s,\phi}(W)
=
\frac{J(s,W,\phi)}{L(s,\pi,\wedge^2)}
\quad\text{if }m=2n,
\qquad
\Sigma_s(W)
=
\frac{J(s,W)}{L(s,\pi,\wedge^2)}
\quad\text{if }m=2n+1.
\]
These are continuous linear functionals on
\(\mathcal{W}(\pi,\psi_F)\), uniformly for \(s\) in compact subsets of
\(\mathbb{C}\); in the even case, this holds for every fixed
\(\phi\in\mathcal{S}(F^n)\).

Moreover, for every \(s_0\in\mathbb{C}\), there exists
\(W_0\in\mathcal{W}(\pi,\psi_F)\), together with
\(\phi_0\in\mathcal{S}(F^n)\) when \(m=2n\), such that
\[
\Sigma_{s_0,\phi_0}(W_0)\neq 0
\quad\text{if }m=2n,
\qquad
\Sigma_{s_0}(W_0)\neq 0
\quad\text{if }m=2n+1.
\]
\end{thm}

Ideally, one would like to use Theorem~\ref{continuity-exterior-square} to classify irreducible generic representations distinguished by the Shalika subgroup. However, an analogue of Proposition~\ref{mirabolic-invariance} does not appear to be available in the literature. We believe that Theorem~\ref{continuity-exterior-square} provides a
necessary starting point for the study of exceptional poles of
archimedean exterior-square \(L\)-factors \cite{CPS1994}.
\end{section}

\begin{section}{Global Applications}\label{s7}

In this section, we combine Theorems~\ref{main1} and~\ref{main2} with
the corresponding non-archimedean results to express the global Asai
and exterior-square \(L\)-functions as finite sums of global Flicker
and Jacquet--Shalika integrals, respectively. We then deduce the
analytic properties of these \(L\)-functions from the global integral
representations, giving proofs independent of the Langlands--Shahidi
method \cite{Shahidi2010}.

We begin by setting up the necessary global notation. Let \(L/K\) be a quadratic extension of number fields. For every place \(v\) of \(K\), we denote by \(K_v\) the corresponding completion of \(K\) and set
\(L_v=K_v\otimes_K L\). Let
\(\mathbb{A}_K=\prod_v' K_v\) and
\(\mathbb{A}_L=\mathbb{A}_K\otimes_K L=\prod_v' L_v\) be the ad\`ele rings of
\(K\) and \(L\), respectively, and let \(\lvert\>\cdot\>\rvert_{\mathbb{A}_K}\)
be the normalized absolute value on \(\mathbb{A}_K\). Define
\(\mathbb{I}_K^1\) as the subgroup of the idele group \(\mathbb{A}_K^\times\)
given by
\[
  \mathbb{I}_K^1
  =
  \bigl\{x\in\mathbb{A}_K^\times:
  \lvert x\rvert_{\mathbb{A}_K}=1\bigr\}.
\]
For the sake of notational brevity, we write
\[
  [G]:=Z(\mathbb{A}_K)G(K)\backslash G(\mathbb{A}_K)
\]
for any reductive group \(G\) with center \(Z\). Let \(\Psi'\) and \(\Psi\)
be nontrivial additive characters of \(K\backslash\mathbb{A}_K\) and
\(L\backslash\mathbb{A}_L\), respectively, with \(\Psi\) being trivial on
\(K\backslash\mathbb{A}_K\). For every place \(v\) of \(K\), let \(\Psi_v'\)
and \(\Psi_v\) be the local components of \(\Psi'\) and \(\Psi\) at \(v\),
respectively. To each, we associate a generic character
\[
  \Psi_{n,v}:N_n(L_v)\longrightarrow\mathbb{S}^1
  \qquad\text{and}\qquad
  \Psi'_{n,v}:N_n(K_v)\longrightarrow\mathbb{S}^1,
\]
as before. Then \(\Psi_n=\prod_v\Psi_{n,v}\) defines a character of
\(N_n(\mathbb{A}_L)\), which is trivial on both \(N_n(L)\) and
\(N_n(\mathbb{A}_K)\), and we define
\(\Psi_n'=\prod_v\Psi'_{n,v}\). For each place \(v\), let \(\tau_v\in L_v^\times\) be the unique element
such that
\[
  \Psi_v(z)
  =
  \Psi'_v\!\left(\operatorname{Tr}_{L_v/K_v}(\tau_v z)\right),
  \qquad z\in L_v.
\]
If \(v\) splits in \(L\), we
identify \(L_v\) with \(K_v\times K_v\); then
$ \tau_v=(\beta_v,-\beta_v) $
for a uniquely determined \(\beta_v\in K_v^\times\) such that
\[
  \Psi_v(x,y)
  =
  \Psi'_v(\beta_vx)\Psi'_v(-\beta_vy),
  \qquad x,y\in K_v.
\]
Let \(\lambda_{L_v/K_v}(\Psi'_v)\) denote the Langlands constant
associated with \(L_v/K_v\) and \(\Psi'_v\), with the convention that
\(\lambda_{L_v/K_v}(\Psi'_v)=1\) when \(v\) splits in \(L\). When \(v\)
is inert, it is a fourth root of unity.

Let \(\mathcal{S}(\mathbb{A}_K^n)\) denote
the Schwartz--Bruhat space on \(\mathbb{A}_K^n\). The Fourier transform on
\(\mathbb{A}_K^n\), denoted by \(\Phi\mapsto\widehat{\Phi}\), is defined as
follows: for every \(\Phi\in\mathcal{S}(\mathbb{A}_K^n)\), we have
\[
  \widehat{\Phi}(x_1,\ldots,x_n)
  =
  \int_{\mathbb{A}_K^n}
  \Phi(y_1,\ldots,y_n)
  \Psi'(x_1y_1+\cdots+x_ny_n)\,dy_1\cdots dy_n,
\]
for all \((x_1,\ldots,x_n)\in\mathbb{A}_K^n\), where the measure of
integration is chosen so that
\(\widehat{\widehat{\Phi}}(v)=\Phi(-v)\). Given a Schwartz--Bruhat function
\(\Phi\in\mathcal{S}(\mathbb{A}_K^n)\), we form the \(\Theta\)-series
\[
  \Theta_\Phi(a,g)
  :=
  \sum_{\xi\in K^n}\Phi(a\xi g),
  \qquad
  a\in\mathbb{A}_K^\times,
  \quad
  g\in G_n(\mathbb{A}_K).
\]

Associated with this \(\Theta\)-series is an Eisenstein series, which is
essentially the Mellin transform of \(\Theta\). To be precise, for a unitary
Hecke character $\eta:K^\times\backslash\mathbb{A}_K^\times\longrightarrow\mathbb{C}^\times,$
we set
\[
  E(g,s;\Phi,\eta)
  :=
  \lvert\det g\rvert_{\mathbb{A}_K}^{s}
  \int_{K^\times\backslash\mathbb{A}_K^\times}
  \Theta_\Phi'(a,g)\eta(a)
  \lvert a\rvert_{\mathbb{A}_K}^{ns}\,d^\times a,
\]
where $\Theta_\Phi'(a,g):=\Theta_\Phi(a,g)-\Phi(0).
$
In \cite{JS1981}, Jacquet and Shalika established the analytic properties of the
Eisenstein series.

\begin{thm}
The Eisenstein series \(E(g,s;\Phi,\eta)\) has a meromorphic continuation to
all of \(\mathbb{C}\). It is entire unless \(\eta\) is trivial on
\(\mathbb{I}_K^1\) of the form
\(\eta(a)=\lvert a\rvert^{in\delta}\), with \(\delta\in\mathbb{R}\), in
which case it has at most simple poles at \(s=-i\delta\) and
\(s=1-i\delta\). As a function of \(g\), it is smooth of moderate growth,
and as a function of \(s\), it is bounded in vertical strips (away from
possible poles), uniformly for \(g\) in compact sets. Moreover, it satisfies
the functional equation
\[
  E(g,s;\Phi,\eta)
  =
  E({}^{t}g^{-1},1-s;\widehat{\Phi},\eta^{-1}).
\]
\end{thm}

\noindent\textbf{Non-archimedean local notation.}
Let \(v\) be a non-archimedean place of \(K\). Let \(\mathcal{O}_{K_v}\) and \(\mathcal{O}_{L_v}\) be the rings of integers of \(K_v\) and \(L_v\), respectively. We use the same representation-theoretic
notation as in the archimedean setting. Thus,
\(\Irr_{\mathrm{gen}}(G_n(K_v))\) denotes the set of isomorphism classes
of irreducible generic representations of \(G_n(K_v)\). For
\(\pi\in\Irr_{\mathrm{gen}}(G_n(K_v))\), we write
\(\mathcal{W}(\pi,\Psi'_{n,v})\), \({\pi}^\vee\), and
\(\omega_\pi\) for its Whittaker model, contragredient, and central
character, respectively.

The global results established in the remainder of this paper extend to arbitrary cuspidal
automorphic representations. Indeed, every cuspidal automorphic
representation $\Pi$ can be written as
$\Pi=\Pi^u\otimes|\det|^r,$
where $\Pi^u$ is a unitary cuspidal automorphic representation and
$r\in\mathbb{R}$.

\subsection{Global Eulerian Flicker Integrals}

Let $(\Pi, V_\Pi)$ be a unitary cuspidal automorphic representation of $G_n(\mathbb{A}_L)$ with central character $\omega_\Pi$. Then $\Pi$ is isomorphic to a restricted tensor product, $\Pi \cong \bigotimes'_v \Pi_v$, taken over the places $v$ of $K$. Here, each $\Pi_v$ belongs to $\mathrm{Irr}_{\mathrm{gen}}(G_n(L_v))$ and is unramified for all but finitely many places $v$. For $\Phi \in \mathcal{S}(\mathbb{A}_K^n)$ and $\varphi \in V_\Pi$, Flicker \cite{Fli1988} defined the global integral
\[
Z(s, \Phi, \varphi) = \int_{[G_n]} E(g, s; \Phi, \omega_\Pi|_{\mathbb{A}_K^\times}) \varphi(g) \, dg,
\]
If $\omega_\Pi|_{\mathbb{A}_K^\times}$ is trivial on $\mathbb{I}_K^1$, let $\delta$ be the real number such that $\omega_\Pi|_{\mathbb{A}_K^\times}(\cdot) = |\cdot|^{in\delta}$. For ease of reference, we collect certain properties of these integrals (\cite{Fli1988}, \cite{Kab2004}).

\begin{prop}\label{global-integral-fe1}
The integral $Z(s, \Phi, \varphi)$ is convergent whenever the Eisenstein series is holomorphic at $s$. It has a meromorphic continuation to the entire complex plane and satisfies the functional equation
\[
Z(s, \Phi, \varphi) = Z(1 - s, \widehat{\Phi}, \widetilde{\varphi}),
\]
where $\widetilde{\varphi}(g) = \varphi(w_n\,{}^tg^{-1})$.
\end{prop}

In \cite{Fli1988}, Flicker showed that the poles of $Z(s, \Phi, \varphi)$ are closely related to those of the Eisenstein series.

\begin{prop}\label{poles-global}
The integral $Z(s, \Phi, \varphi)$ is entire if $\omega_\Pi$ is nontrivial on $\mathbb{I}_K^1$. Otherwise, it has at most simple poles at $s = -i\delta$ and $s = 1 - i\delta$.
\end{prop}

We recall the factorization of the integral $Z(s, \Phi, \varphi)$ (as in Sections 2 and 3 of \cite{Fli1988}).

\begin{prop}\label{global-duo1}
If $\varphi \in V_\Pi$ is a cusp form, let
\[
W_\varphi(g) = \int_{N_n(L)\backslash N_n(\mathbb{A}_L)} \varphi(ng)\overline{\Psi}(n) \, dn
\]
be the associated Whittaker function. For $\Phi \in \mathcal{S}(\mathbb{A}_K^n)$, the integral
\[
I(s, W_\varphi, \Phi) = \int_{N_n(\mathbb{A}_K)\backslash G_n(\mathbb{A}_K)} W_\varphi(g)\Phi(e_n g)|\det(g)|_{\mathbb{A}_K}^s \, dg
\]
converges absolutely and uniformly on compact sets when $\operatorname{Re}(s)$ is sufficiently large. When this is the case, we have
\[
Z(s, \Phi, \varphi) = I(s, W_\varphi, \Phi).
\]
\end{prop}

For decomposable data, the global integrals factor as products of local
integrals. Let \(W_\varphi=\prod_v W_v\), where \(v\) runs over all
places of \(K\) and \(W_v\in\mathcal{W}(\Pi_v,\Psi_{n,v})\). For almost
all unramified places \(v\), we take \(W_v\) to be the normalized
spherical Whittaker function, that is, the unique
\(G_n(\mathcal{O}_{L_v})\)-invariant Whittaker function satisfying
\(W_v(1)=1\).

Similarly, let $\Phi = \prod_v \Phi_v$, where each $\Phi_v$ is a Schwartz function in $\mathcal{S}(K_v^n)$, and for almost all unramified places $v$, $\Phi_v$ is the characteristic function of $\mathcal{O}_{K_v}^n$. When $\operatorname{Re}(s)$ is sufficiently large,
\[
I(s, W_\varphi, \Phi) = \prod_v I(s, W_v, \Phi_v),
\]
where
\[
I(s, W_v, \Phi_v) = \int_{N_n(K_v)\backslash G_n(K_v)} W_v(g)\Phi_v(e_n g)|\det(g)|_{K_v}^s \, dg.
\]

When a place $v$ of $K$ splits in $L$, this local integral coincides with the Rankin-Selberg integral \cite{JPSS1983}. 

\subsection{Global Asai L-function}

Let \(S_\infty\) denote the set of archimedean places of \(K\), and let
\(S\) be a finite set of places of \(K\) containing \(S_\infty\), all
finite places ramified in \(L/K\), and all finite places \(v\) at which
\(\Pi_v\), \(\Psi_v\), or \(\Psi'_v\) is ramified.

For finite places $v$, the local factors $L(s,\Pi_v, \As)$ (denoting the Rankin--Selberg $L$-function if $v$ splits over $L$) and $\varepsilon(s,\Pi_v, \As,\Psi'_{v})$ are defined via Flicker integrals \cite{Fli1993,Kab2004,JPSS1983}. For archimedean places, they are defined via the local Langlands correspondence (Section~\ref{s2}). We formally define the global Asai $L$- and $\varepsilon$-factors as
\[
\begin{aligned}
L(s,\Pi,\As)
&=\prod_v L(s,\Pi_v,\As),\\
\varepsilon(s,\Pi,\As)
&=\prod_v\varepsilon(s,\Pi_v,\As,\Psi'_{v}).
\end{aligned}
\]
We recall the unramified computation of the local Flicker integrals. \cite{JPSS1983} addresses the split case, while \cite{Fli1988} covers the inert case. For a more recent exposition, see \cite{BP2021}.

\begin{prop}\label{unramified1}
Let \(v\notin S\). Let
\(W_v^\circ\in\mathcal W(\Pi_v,\Psi_{n,v})\)
be the normalized spherical Whittaker function, and let
\(\Phi_v^\circ\) be the characteristic function of
\(\mathcal O_{K_v}^n\). Then
\[
I(s,W_v^\circ,\Phi_v^\circ)=L(s,\Pi_v,\As).
\]
\end{prop}

Therefore, the partial Asai $L$-function
\[L^{S}(s,\Pi,\As)=\prod_{v\notin S} L(s,\Pi_v,\As)=\prod_{v\notin S} I(s, W_v, \Phi_v)\]
is absolutely convergent for $\operatorname{Re}(s)\gg0$. Thus, the same is true for $L(s,\Pi,\As)$.

Also from local calculations in \cite{Fli1993,Kab2004,JPSS1983, BP2021} and local functional equation, $$\varepsilon(s,\Pi_v,\As,\Psi'_{v})=1\>\> \text{for}\>\>v\notin S.$$
Therefore, \(\epsilon(s,\Pi,\As)\) is a finite product. We now express the global Asai \(L\)-function as a finite sum of global
Flicker integrals.

\begin{thm}\label{sum-global}
There exist an integer $r \geq 1$ and finite collections $\{\varphi_i\}_{i=1}^r \subset V_\Pi$ and $\{\Phi_i\}_{i=1}^r \subset \mathcal{S}(\mathbb{A}_K^n)$ such that
\[
  L(s, \Pi, \As) = \sum_{i=1}^r Z(s, \Phi_i, \varphi_i).
\]
\end{thm}

\begin{proof}
For every finite place $v \in S$, \cite[Theorem~5.3]{Mat2011} and \cite{JPSS1983} ensure the existence of finite collections $\{W_{v,i}\}$ and $\{\Phi_{v,i}\}$ such that
\[
  L(s, \Pi_v, \As) = \sum_i I(s, W_{v,i}, \Phi_{v,i}).
\]
By Theorem~\ref{main1} and \cite[Theorem~2.7]{Jac2009}, this holds for every $v \in S_\infty$ as well. At every \(v\notin S\), choose the unramified data from
Proposition~\ref{unramified1}. Multiplying the local finite-sum identities over
\(v\in S\) and expanding the resulting finite product, we obtain
finitely many decomposable global data. Proposition~\ref{global-duo1} then gives
the desired identity.
\end{proof}

We now deduce the following
analytic properties of the global Asai \(L\)-function.

\begin{corollarySubsec}
\label{global-asai-properties}
The global Asai \(L\)-function \(L(s,\Pi,\As)\) has the following
analytic properties:
\begin{enumerate}
    \item It admits meromorphic continuation to \(\mathbb{C}\).

    \item Its meromorphic continuation is bounded at infinity in
    vertical strips.

    \item It satisfies the functional equation
    \[
    L(s,\Pi,\As)
    =
    \varepsilon(s,\Pi,\As)
    L(1-s,\Pi^\vee,\As).
    \]

    \item If
    \(\left.\omega_\Pi\right|_{\mathbb{A}_K^\times}\)
    is nontrivial on \(\mathbb{I}_K^1\), then
    \(L(s,\Pi,\As)\) is entire. Otherwise, it has at most simple poles
    at \(s=-i\delta\) and \(s=1-i\delta\).
\end{enumerate}
\end{corollarySubsec}

\begin{proof}
By Theorem~\ref{sum-global}, \(L(s,\Pi,\As)\) is a finite sum of
global Flicker integrals. Its meromorphic continuation therefore
follows from that of \(Z(s,\Phi,\varphi)\). The vertical-strip estimate
for the Eisenstein series, together with the rapid decay of cusp forms,
shows that every global Flicker integral is bounded at infinity in
vertical strips away from its possible poles. The same is consequently
true of \(L(s,\Pi,\As)\).

If
\(\left.\omega_\Pi\right|_{\mathbb{A}_K^\times}\)
is nontrivial on \(\mathbb{I}_K^1\), then every summand in
Theorem~\ref{sum-global} is entire by
Proposition~\ref{poles-global}, and hence \(L(s,\Pi,\As)\) is entire.
Otherwise, each summand has at most simple poles at
\(s=-i\delta\) and \(s=1-i\delta\). The same pole bound therefore holds
for their finite sum. This proves assertions~(1), (2), and~(4).

It remains to prove the functional equation. For every place \(v\), we
use the notation
\[
\Lambda_{s,\Phi_v}^{\Pi_v}(W_v)
=
\frac{I(s,W_v,\Phi_v)}
     {L(s,\Pi_v,\As)}.
\]
For each \(v\in S\), choose
$W_v\in\mathcal{W}(\Pi_v,\Psi_{n,v})$
 and 
$\Phi_v\in\mathcal{S}(K_v^n)$
such that
\(\Lambda_{s,\Phi_v}^{\Pi_v}(W_v)\) is not identically zero. Such data
exist by the local results used in
Proposition~\ref{sum-global}. At every \(v\notin S\), take the
normalized unramified data.

Choose decomposable global data \(\Phi=\prod_v\Phi_v\) and
\(\varphi\in V_\Pi\) such that
\(W_\varphi=\prod_vW_v\). The global functional equation (Proposition \ref{global-integral-fe1}) and the
Euler factorizations give, initially in a right half-plane and hence
meromorphically on \(\mathbb{C}\),
\[
\begin{aligned}
\left(
\prod_{v\in S}
\Lambda_{s,\Phi_v}^{\Pi_v}(W_v)
\right)
L(s,\Pi,\As)
&=
Z(s,\Phi,\varphi)\\
&=
Z(1-s,\widehat{\Phi},\widetilde{\varphi})\\
&=
\left(
\prod_{v\in S}
\Lambda_{1-s,\widehat{\Phi}_v}^{\Pi_v^\vee}
(\widetilde W_v)
\right)
L(1-s,\Pi^\vee,\As).
\end{aligned}
\]

Put \(N=n(n-1)/2\) and, for every place \(v\), set
\[
c_v(s)
=
\omega_{\Pi_v}(\tau_v)^{n-1}
|\tau_v|_{L_v}^{N(s-\frac12)}
\lambda_{L_v/K_v}(\Psi'_v)^{-N}.
\]
The normalized local functional equation gives
\[
\Lambda_{1-s,\widehat{\Phi}_v}^{\Pi_v^\vee}
(\widetilde W_v)
=
c_v(s)\,
\varepsilon(s,\Pi_v,\As,\Psi'_v)\,
\Lambda_{s,\Phi_v}^{\Pi_v}(W_v).
\]
Since
$\prod\limits_{v\in S}
\Lambda_{s,\Phi_v}^{\Pi_v}(W_v)
$
is not identically zero, we may cancel it from the preceding
meromorphic identity. We obtain
\[
L(s,\Pi,\As)
=
\left(
\prod_{v\in S}
c_v(s)\,
\varepsilon(s,\Pi_v,\As,\Psi'_v)
\right)
L(1-s,\Pi^\vee,\As).
\]

For \(v\notin S\), the normalized local integrals on both sides of the
local functional equation are equal to \(1\). Hence
\[
c_v(s)\,
\varepsilon(s,\Pi_v,\As,\Psi'_v)
=
1.
\]
It follows that the product over \(v\in S\) may be replaced by the
product over all places. Moreover, the global product formula
\cite[Proof of Theorem~3.9.1, p.~40]{BP2021} gives
\[
\prod_v c_v(s)=1.
\]
Consequently,
\[
L(s,\Pi,\As)
=
\left(
\prod_v
\varepsilon(s,\Pi_v,\As,\Psi'_v)
\right)
L(1-s,\Pi^\vee,\As)
=
\varepsilon(s,\Pi,\As)
L(1-s,\Pi^\vee,\As),
\]
which is the desired functional equation.
\end{proof}

\subsection{Global Eulerian Jacquet--Shalika Integrals}

Let $m \geq 2$ be an integer, which we write as $m = 2n$ or $m = 2n+1$ depending on its parity. Let \((\Pi,V_\Pi)\) be a unitary cuspidal automorphic
representation of \(G_m(\mathbb{A}_K)\), with central character \(\omega_\Pi\). Then $\Pi\cong\bigotimes_v'\Pi_v,$
where $
  \Pi_v\in\operatorname{Irr}_{\mathrm{gen}}(G_m(K_v))
$ and \(\Pi_v\) is unramified for all but finitely many places \(v\). For \(\varphi\in V_\Pi\), let
\[
  W_\varphi(g)
  =
  \int_{N_m(K)\backslash N_m(\mathbb{A}_K)}
  \varphi(ug)\overline{\Psi_m'(u)}\,du
\]
be the associated Whittaker function. If $\omega_\Pi|_{\mathbb{A}_K^\times}$ is trivial on $\mathbb{I}_K^1$, let $\delta$ be the real number such that $\omega_\Pi|_{\mathbb{A}_K^\times}(\cdot) = |\cdot|^{in\delta}$.

\medskip
\noindent\textbf{The even case.}
Suppose that \(m=2n\). For \(\Phi\in\mathcal{S}(\mathbb{A}_K^n)\) and
\(\varphi\in V_\Pi\), Jacquet and Shalika defined
\cite{JS1990}
\[
\begin{split}
  I_{\mathrm{JS}}^{\mathrm{even}}(s,\Phi,\varphi)
  :=
  \int_{[G_n]}
  &E(g,s;\Phi,\omega_\Pi)                                    \\
  &\times
  \int_{M_n(K)\backslash M_n(\mathbb{A}_K)}
  \varphi\left(
    \begin{pmatrix}
      I_n & X\\
      0   & I_n
    \end{pmatrix}
    \begin{pmatrix}
      g & 0\\
      0 & g
    \end{pmatrix}
  \right)
  \Psi'\bigl(\operatorname{tr}X\bigr)\,dX\,dg.
\end{split}
\]

\begin{prop}\label{prop:global-js-even-analytic}
The integral
\(I_{\mathrm{JS}}^{\mathrm{even}}(s,\Phi,\varphi)\)
converges absolutely and locally uniformly at every point at which
\(E(g,s;\Phi,\omega_\Pi)\) is holomorphic. It admits a meromorphic
continuation to \(\mathbb{C}\).
\end{prop}

\begin{prop}\label{prop:global-js-even-poles}
If \(\omega_\Pi\) is nontrivial on \(\mathbb{I}_K^1\), then
\(I_{\mathrm{JS}}^{\mathrm{even}}(s,\Phi,\varphi)\) is entire.
If \(\omega_\Pi\) is trivial on \(\mathbb{I}_K^1\), it has at most
simple poles at $s=-i\delta
  \>\text{and}\>
  s=1-i\delta.$
\end{prop}

The global integral unfolds to an Eulerian integral. Define
\[
\begin{split}
  J_{\mathrm{JS}}^{\mathrm{even}}
  (s,W_\varphi,\Phi)
  :=
  \int_{N_n(\mathbb{A}_K)\backslash G_n(\mathbb{A}_K)}
  \int_{V_n(\mathbb{A}_K)\backslash M_n(\mathbb{A}_K)}
  &W_\varphi\left(
    \sigma_{2n}
    \begin{pmatrix}
      I_n & X\\
      0   & I_n
    \end{pmatrix}
    \begin{pmatrix}
      g & 0\\
      0 & g
    \end{pmatrix}
  \right)                                                     \\
  &\times
  \Psi'\bigl(-\operatorname{tr}X\bigr)
  \Phi(e_ng)
  |\det g|_{\mathbb{A}_K}^{s}\,dX\,dg.
\end{split}
\]

\begin{prop}\label{prop:global-js-even-unfolding}
The integral
\(J_{\mathrm{JS}}^{\mathrm{even}}(s,W_\varphi,\Phi)\)
converges absolutely and uniformly on compact subsets when
\(\operatorname{Re}(s)\) is sufficiently large. In this region,
\[
  I_{\mathrm{JS}}^{\mathrm{even}}(s,\Phi,\varphi)
  =
  J_{\mathrm{JS}}^{\mathrm{even}}(s,W_\varphi,\Phi).
\]
\end{prop}

Suppose that the data are decomposable:
\[
  W_\varphi=\prod_v W_v,
  \qquad
  \Phi=\prod_v\Phi_v,
\]
where $ W_v\in\mathcal{W}(\Pi_v,\Psi_{2n,v}')
  \>\text{and}\>
  \Phi_v\in\mathcal{S}(K_v^n).
$
At almost all non-archimedean places \(v\), let \(W_v=W_v^\circ\)
be the normalized spherical Whittaker function, and let  $\Phi_v=\Phi_v^\circ$
be the characteristic function on $\mathcal{O}_{K_v}^n$.
Then, when \(\operatorname{Re}(s)\) is sufficiently large,
\[
  J_{\mathrm{JS}}^{\mathrm{even}}(s,W_\varphi,\Phi)
  =
  \prod_v
  J(s,W_v,\Phi_v),
\]
where
\[
\begin{split}
  J(s,W_v,\Phi_v)
  :=
  \int_{N_n(K_v)\backslash G_n(K_v)}
  \int_{V_n(K_v)\backslash M_n(K_v)}
  &W_v\left(
    \sigma_{2n}
    \begin{pmatrix}
      I_n & X\\
      0   & I_n
    \end{pmatrix}
    \begin{pmatrix}
      g & 0\\
      0 & g
    \end{pmatrix}
  \right)                                                     \\
  &\times
  \Psi_v'\bigl(-\operatorname{tr}X\bigr)
  \Phi_v(e_ng)
  |\det g|_{K_v}^{s}\,dX\,dg.
\end{split}
\]

\medskip
\noindent\textbf{The odd case.}
Suppose that \(m=2n+1\). For \(\varphi\in V_\Pi\), Jacquet and
Shalika defined
\cite{JS1990}; see also \cite[Section~6]{Bel2011},
\[
\begin{split}
  I_{\mathrm{JS}}^{\mathrm{odd}}(s,\varphi)
  :=
  \int_{G_n(K)\backslash G_n(\mathbb{A}_K)}
  \int_{M_n(K)\backslash M_n(\mathbb{A}_K)}
  \int_{K^n\backslash\mathbb{A}_K^n}
  &\varphi\left(
    \begin{pmatrix}
      I_n & X   & Y\\
      0   & I_n & 0\\
      0   & 0   & 1
    \end{pmatrix}
    \begin{pmatrix}
      g & 0 & 0\\
      0 & g & 0\\
      0 & 0 & 1
    \end{pmatrix}
  \right)                                                     \\
  &\times
  \Psi'\bigl(\operatorname{tr}X\bigr)
  |\det g|_{\mathbb{A}_K}^{s-1}\,dY\,dX\,dg,
\end{split}
\]
where \(Y\) is regarded as a column vector.

\begin{prop}\label{prop:global-js-odd-analytic}
The integral \(I_{\mathrm{JS}}^{\mathrm{odd}}(s,\varphi)\)
converges absolutely and locally uniformly for every
\(s\in\mathbb{C}\). In particular, it defines an entire function
of \(s\).
\end{prop}

Define the Eulerian integral
\[
\begin{split}
  J_{\mathrm{JS}}^{\mathrm{odd}}(s,W_\varphi)
  :=
  \int_{N_n(\mathbb{A}_K)\backslash G_n(\mathbb{A}_K)}
  \int_{V_n(\mathbb{A}_K)\backslash M_n(\mathbb{A}_K)}
  &W_\varphi\left(
    \sigma_{2n+1}
    \begin{pmatrix}
      I_n & X   & 0\\
      0   & I_n & 0\\
      0   & 0   & 1
    \end{pmatrix}
    \begin{pmatrix}
      g & 0 & 0\\
      0 & g & 0\\
      0 & 0 & 1
    \end{pmatrix}
  \right)                                                     \\
  &\times
  \Psi'\bigl(-\operatorname{tr}X\bigr)
  |\det g|_{\mathbb{A}_K}^{s-1}\,dX\,dg.
\end{split}
\]

\begin{prop}\label{prop:global-js-odd-unfolding}
The integral
\(J_{\mathrm{JS}}^{\mathrm{odd}}(s,W_\varphi)\)
converges absolutely and uniformly on compact subsets when
\(\operatorname{Re}(s)\) is sufficiently large. In this region,
\[
  I_{\mathrm{JS}}^{\mathrm{odd}}(s,\varphi)
  =
  J_{\mathrm{JS}}^{\mathrm{odd}}(s,W_\varphi).
\]
\end{prop}

Suppose that
\[
  W_\varphi=\prod_v W_v,
  \qquad
  W_v\in\mathcal{W}(\Pi_v,\Psi_{2n+1,v}').
\]
At almost all non-archimedean places \(v\), let \(W_v=W_v^\circ\)
be the normalized spherical Whittaker function. Then, when
\(\operatorname{Re}(s)\) is sufficiently large,
\[
  J_{\mathrm{JS}}^{\mathrm{odd}}(s,W_\varphi)
  =
  \prod_v
  J(s,W_v),
\]
where
\[
\begin{split}
  J(s,W_v)
  :=
  \int_{N_n(K_v)\backslash G_n(K_v)}
  \int_{V_n(K_v)\backslash M_n(K_v)}
  &W_v\left(
    \sigma_{2n+1}
    \begin{pmatrix}
      I_n & X   & 0\\
      0   & I_n & 0\\
      0   & 0   & 1
    \end{pmatrix}
    \begin{pmatrix}
      g & 0 & 0\\
      0 & g & 0\\
      0 & 0 & 1
    \end{pmatrix}
  \right)                                                     \\
  &\times
  \Psi_v'\bigl(-\operatorname{tr}X\bigr)
  |\det g|_{K_v}^{s-1}\,dX\,dg.
\end{split}
\]

\subsection{Global Exterior-Square \(L\)-function}

Let \(S\) be a finite set of places of \(K\), containing \(S_\infty\),
such that, for every \(v\notin S\), the representation \(\Pi_v\) and
the additive character \(\Psi'_v\) are unramified.

For each non-archimedean place \(v\), the local factors
\[
  L(s,\Pi_v,\wedge^2)
  \qquad\text{and}\qquad
  \varepsilon(s,\Pi_v,\wedge^2,\Psi_v')
\]
are defined using the local theory of Jacquet--Shalika integrals
\cite{JS1990,KR2012,Mat2014,CM2015}. For each archimedean place, the local factors
\(L(s,\Pi_v,\wedge^2)\) and
\(\varepsilon(s,\Pi_v,\wedge^2,\Psi_v')\) are as in Section~\ref{s2}. We formally define
\[
  L(s,\Pi,\wedge^2)
  =
  \prod_v L(s,\Pi_v,\wedge^2)
\]
\[
  \varepsilon(s,\Pi,\wedge^2)
  =
  \prod_v
  \varepsilon(s,\Pi_v,\wedge^2,\Psi_v').
\]

We recall the unramified computation of the local
Jacquet--Shalika integrals.

\begin{prop}\label{unramified-exterior-square}
Let \(v\notin S\). Then,
\[L(s,\Pi_v,\wedge^2)=
\begin{cases}
 J(s,W_v^\circ,\Phi_v^\circ) & \text{if } m=2n,\\[6pt]
 J(s,W_v^\circ) & \text{if } m=2n+1.
\end{cases}
\]
\end{prop}

Therefore, the partial exterior-square \(L\)-function
\[
  L^S(s,\Pi,\wedge^2)
  =
  \prod_{v\notin S}L(s,\Pi_v,\wedge^2)
\]
is absolutely convergent for \(\operatorname{Re}(s)\gg0\). Thus,
the same is true for \(L(s,\Pi,\wedge^2)\).

Also, the unramified local calculations and the local functional
equations \cite{JS1990, KR2012, CM2015,Mat2014} give
\[
  \varepsilon(s,\Pi_v,\wedge^2,\Psi_v')
  =
  1
  \qquad\text{for every }v\notin S.
\]
Therefore, \(\varepsilon(s,\Pi,\wedge^2)\) is a finite product. We now express the global exterior-square \(L\)-function as a finite sum
of global Jacquet--Shalika integrals.

\begin{thm}\label{sum-global-exterior-square}
There exist an integer \(r\geq 1\) and a finite family
\(\{\varphi_i\}_{i=1}^r\subset V_\Pi\) (along with Schwartz functions
\(\{\Phi_i\}_{i=1}^r\subset\mathcal{S}(\mathbb{A}_K^n)\) when \(m=2n\))
such that
\[
L(s,\Pi,\wedge^2)
=
\begin{cases}
\displaystyle
\sum_{i=1}^r
I_{\mathrm{JS}}^{\mathrm{even}}(s,\Phi_i,\varphi_i)
& \text{if } m=2n,\\[6pt]
\displaystyle
\sum_{i=1}^r
I_{\mathrm{JS}}^{\mathrm{odd}}(s,\varphi_i)
& \text{if } m=2n+1.
\end{cases}
\]
\end{thm}

\begin{proof}
We treat the case \(m=2n\); the odd case is similar. For each finite
place \(v\in S\), the non-archimedean local theory
\cite{KR2012,Jo2020} provides finite families
\(\{W_{v,i}\}_i\subset\mathcal{W}(\Pi_v,\Psi'_{m,v})\) and
\(\{\Phi_{v,i}\}_i\subset\mathcal{S}(K_v^n)\) such that
\[
L(s,\Pi_v,\wedge^2)
=
\sum_i J(s,W_{v,i},\Phi_{v,i}).
\]
For each archimedean place \(v\in S_\infty\), the same conclusion follows
from Theorem~\ref{main2}. At every \(v\notin S\), choose the unramified data from
Proposition~\ref{unramified-exterior-square}. Multiplying the local finite-sum identities over \(v\in S\) and expanding the resulting finite product, we obtain
finitely many decomposable global data. The result now follows from
Proposition~\ref{prop:global-js-even-unfolding} in the even case and Proposition~\ref{prop:global-js-odd-unfolding} in the odd
case.
\end{proof}

The following result is the exterior-square analogue of
Corollary~\ref{global-asai-properties}. It follows from
Theorem~\ref{sum-global-exterior-square} by the same argument, and we
therefore omit the proof.

\begin{corollarySubsec} 
\label{global-exterior-square-properties}
The global exterior-square \(L\)-function
\(L(s,\Pi,\wedge^2)\) has the following analytic properties:
\begin{enumerate}
  \item
  \(L(s,\Pi,\wedge^2)\) admits meromorphic continuation to
  \(\mathbb{C}\).

  \item
  Its meromorphic continuation is bounded at infinity in vertical
  strips.

  \item
  It satisfies the functional equation
  \[
    L(s,\Pi,\wedge^2)
    =
    \varepsilon(s,\Pi,\wedge^2)
    L(1-s,{\Pi}^\vee,\wedge^2).
  \]

  \item
  If \(m=2n+1\), then \(L(s,\Pi,\wedge^2)\) is entire. If \(m=2n\), then \(L(s,\Pi,\wedge^2)\) is entire whenever
  \(\omega_\Pi\) is nontrivial on \(\mathbb{I}_K^1\). Otherwise, it has at most simple poles at $s=-i\delta\>\text{and}\>
    s=1-i\delta.$
\end{enumerate}
\end{corollarySubsec}

\end{section}

\begin{acknowledgements}
The authors express their sincere gratitude to U. K. Anandavardhanan, T. Ishii, N. Matringe, A. K. Mondal, S. Nadimpalli, D. Prasad, and R. Raghunathan for their insightful comments and questions. Y. Jo 
 is also indebted to J. Cogdell for bringing to his attention an unfinished project concerning the archimedean exterior-square $L$-factors \cite{CPS1994} many years ago, 
and J. Girsch and R. Kurinczuk for their encouragement and their interest in this work. Y. Jo would like to thank the LMS-Sheffield Symposia for support and hospitality during the symposium ``The Langlands Programme:
Recent Trends, New Developments, and Applications" where work on this paper was undertaken.
\end{acknowledgements}

\begin{funding}
Y. Jo and A. Yadav were partially supported by Global-Learning \& Academic research institution for Master’s·PhD students, and Postdocs (G-LAMP) Program of the National Research Foundation of Korea (NRF) grant funded by the Ministry of Education (No. RS-2025-25442252).
Y. Jo was also partially supported by the NRF grant 
funded by the Korea government (No. RS-2023-00209992). 
\end{funding}

\end{document}